\documentclass[times,sort&compress,3p]{elsarticle}
\journal{Journal of Multivariate Analysis}
\usepackage[labelfont=bf]{caption}

\usepackage{amsmath,amsfonts,amssymb,amsthm,booktabs,color,graphicx,hyperref,url}
\usepackage{dsfont}
\usepackage{enumitem}
\usepackage{subfig}
\usepackage[final]{pdfpages}

\newcommand{\as}{\quad \mbox{as} \quad }
\newcommand{\aand}{\quad \mbox{and} \quad }
\newcommand{\bs}{\boldsymbol}

\newcommand{\eps}{\varepsilon}
\newcommand{\Eps}{\mathcal{E}}
\newcommand{\E}{{\rm E}}

\newcommand{\ka}{\mbox{\Large $\kappa$}}
\newcommand{\K}{\mathcal{K}}
\newcommand{\lr}[1]{\langle #1 \rangle}
\newcommand{\mR}{\bs{\mathcal{R}}}
\newcommand{\mQ}{\bs{\mathcal{Q}}}
\newcommand{\mX}{\bs{\mathcal{X}}}
\newcommand{\N}{\mathds{N}}
\newcommand{\R}{\mathds{R}}
\newcommand{\simiid}{\underset{iid}{\sim}}

\newcommand{\To}{\longrightarrow}
\newcommand{\Z}{\mathds{Z}}

\def\ind{\perp\!\!\!\!\perp}

\newcommand{\corr}{{\rm Corr}}
\newcommand{\cov}{{\rm Cov}}
\newcommand{\var}{{\rm Var}}
\newcommand{\trace}{{\rm tr}}

\newcommand{\sign}{{\rm sign}}
\newcommand{\vc}{{\rm vec}}

\graphicspath{{figures1/}}

\theoremstyle{definition}

\newtheorem{remark}{Remark}[section]

\theoremstyle{plain}
\newtheorem{prop}{Proposition}[section]
\newtheorem{thm}{Theorem}[section]
\newtheorem{lemma}{Lemma}[section]
\newtheorem{coro}{Corollary}[section]

\begin{document}

\begin{frontmatter}

\title{On the behavior of the DFA and DCCA in trend-stationary processes}

\author[A1]{Taiane Schaedler Prass\corref{mycorrespondingauthor}}
\author[A1]{Guilherme Pumi}

\address[A1]{Instituto de Matem\'atica e Estat\'istica - Universidade Federal do Rio Grande do Sul - Av. Bento Gon\c calves, 9500, Porto Alegre - RS, Brazil. }

\cortext[mycorrespondingauthor]{Corresponding author. Email address: \url{taiane.prass@ufrgs.br}}

\begin{abstract}
In this work, we develop the asymptotic theory of the Detrended Fluctuation Analysis (DFA) and Detrended Cross-Correlation Analysis (DCCA) for trend-stationary stochastic processes without any assumption on the specific form of the underlying distribution.  All results are presented and derived under the general framework of potentially overlapping boxes for the polynomial fit.
We prove the stationarity of the DFA and DCCA, viewed as stochastic processes, obtain closed forms for moments up to second order, including the covariance structure for DFA and DCCA and a miscellany of law of large number related results. Our results generalize and improve several results presented in the literature. To verify the behavior of our theoretical results in small samples, we present a Monte Carlo simulation study and an empirical application to econometric time series.
\end{abstract}

\end{frontmatter}

\section{Introduction}

Calculating basic statistics as variance, correlation, cross-correlation among others from non-stationary data is a challenging problem. In this context, commonly applied statistics such as sample autocorrelation, sample variance and sample cross-correlation lose their traditional meaning. Given the importance of such quantities, circumventing this problem becomes an essential matter.

The Detrended Fluctuation Analysis (DFA), introduced by \cite{Peng1994}, is often heuristically described as an indirect way to quantifying variation in trend-stationary time series (understood as the sum of a stationary process plus a polynomial trend).  A generalization of the DFA for the context where the interest lies in the joint behavior of a pair of time series is the Detrended Cross-Correlation Analysis (DCCA), introduced  by \cite{Z2011} based on the detrended covariance of \cite{PS2008} and the DFA. In this sense, the DCCA is an indirect quantifier of cross-correlation.

{Applications of the DFA and DCCA are abundant. A list of near 100 applications of DFA is presented in \cite{Kantelhardt2001,Hu2001} and references therein. Applications include areas such as physics, medicine, economics, bioinformatics, meteorology, among others. In applications, the DFA is employed as a tool to detect and quantify long-range dependence in trend-stationary time series. As for the DCCA, it is usually applied as a measure of long-range cross-correlation between time series. For instance, applications of the DFA and DCCA in economics include the study of the correlation and cross-correlation structure between Brazilian stock and commodity markets \citep{Siqueira2010} and between  Chinese A-share and B-share markets \citep{Wang2010}. In the modeling of traffic flow, the DFA and DCCA are applied to the study of long-range correlation and cross-correlation between different traffic fluctuations signals \citep{Xu2010}. Recently, \cite{Zebende2018} applies the DFA and DCCA to analyze the relationship between relative air humidity and temperature. Other applications of the DCCA can be found in \cite{Petal2011} and reference therein.}

In the literature, both DFA and DCCA, are usually defined in a {constructively fashion based on a sample from a given stochastic process,  that is, as estimators}. Interestingly, the literature is remarkably vague about their theoretical counterparts. Instead, the focus usually lies on the relationship between the DFA/DCCA and the underlying time series, especially in the context of long-range dependence non-stationary time series, which are the core of applications of these methodologies. In other words, what does the DCCA and DFA measure is still unknown. In this paper, we make an effort to solve this problem by investigating and giving meaning to the DFA and DCCA theoretical counterparts. Incidentally, the precise definition of the DFA and DCCA's theoretical counterparts will open the possibility of discussion about the properties of the DFA and DCCA as estimators, such as consistency and unbiasedness, absent in the current literature.

Large sample results for the DFA and DCCA are known under restrictions on the underlying process. For the DFA and DCCA, \cite{BK2008} presents large sample results in the context of fractional Gaussian noise and fractional Brownian Motion. For the DCCA, asymptotic theory is available only for long-range dependent trend-stationary time series {decomposable} as a sum of a polynomial trend plus a fractional Gaussian noise \citep{Blythe2013,Blythe2016}. To the best of our knowledge, large sample results under the general scope of stationary processes are not available. In time series, long-range dependence is often regarded as a {complicated} and delicate subject, especially when compared to classical ARMA processes. So a fair question is: why the asymptotic theory for the DCCA and DFA {is} established under non-stationary fractional Gaussian noise assumptions? We can think of three good reasons for that. First, the DFA and DCCA were designed with long-range dependent data in mind. So it is only natural that the theory has been developed {in} this direction. Second, mathematically speaking, the context of fractional Gaussian noise (or fractional Brownian motion) is very convenient not only because it presents a well-developed theory, but also because when working with DFA/DCCA, it entails several simplifications that hold specifically for these processes, but not for general stationary time series. Third, adding a polynomial trend to the base stationary process allows working in the more general context of trend-stationary time series. {Such a trend, however, does not affect either the DFA or the DCCA, in such a way that results valid for stationary processes will also be valid for} trend-stationary time series, without any modification.

In this work, we are interested in developing the theory of DCCA for jointly trend-stationary processes. As we shall show, it is sufficient to work {in the context of jointly stationary processes because deterministic polynomial trends play no role} in the asymptotic theory. Some of the established literature consider non-overlapping boxes to calculate the DFA and DCCA. This approach {allows for} some simplifications, but there is no theoretical reason {for doing so} and, in practice, applying overlapping boxes can be advantageous, especially {in} small sample sizes. Hence, we shall consider the more general framework of potentially overlapping boxes. In this work, we shall derive several results regarding DFA and DCCA under stationarity conditions and the existence of the appropriate moments. Among these results, {we highlight the stationarity of the DFA and DCCA, derivation of closed forms for the moments up to second-order (including the covariance structure for DFA and DCCA) as well as} a miscellany of law of large number related results. {We also present a Monte Carlo simulation study and an empirical application}.

The paper is organized as follows. In the next section, we define the DFA and DCCA and introduce some notation. Section~\ref{sec:3} is concerned with stationarity results for the DCCA and DFA and the study of the DFA and DCCA's theoretical counterparts. The covariance structure of the DCCA and DFA (as stochastic processes) and a miscellany of law of large number results are derived. In Section~\ref{special} we discuss some special cases of general interest. The proofs of all results are presented in the Appendix. This paper accompanies a supplementary material, which contains further examples, a Monte Carlo simulation study and also an empirical application of the DCCA to the study of the joint behavior of 4 stock indexes (S\&P500, Nasdaq, Dow Jones and Ibovespa) and the Bitcoin cryptocurrency.


\section{Detrended Cross-Correlation Coefficient}

Throughout this paper, given a sequence $\{Y_{k,t}\}_{t=1}^n$, let $\bs{Y}_{k,j}^{\lr{i}}$ be defined by
\[
\bs{Y}_{k,j}^{\lr{i}} =  (Y_{k,i}, \dots, Y_{k,j})^\top  , \quad   \ i,j \in\{ 1, \dots, n\}, \  i \leq j.
 \]
{ For any $\ell \times \ell$ matrix  $A_\ell$, let  $A_\ell^{\lr{i}}$ be the matrix containing the elements of $A_\ell$, from row $i$ up to row $\ell$. Given a block matrix $A$,  let $[A]^{p,q}$ denote its $(p,q)$th block.  Also,  let $0_n$ and $1_n$  denote vectors of zeros and ones in $\R^n$,  respectively,  $0_{m,n}$ and $1_{m,n}$ denote $m\times n$ matrices of zeros and ones, respectively,  and $I_n$  denote the $n\times n$ identity matrix. }

Let $\{X_{1,t}\}_{t\in\Z}$ and  $\{X_{2,t}\}_{t\in\Z}$ be two stochastic processes {and let $\{X_{1,t}\}_{t=1}^n$ and  $\{X_{2,t}\}_{t=1}^n$  be two samples of size $n$ obtained from these processes, respectively.}
Define the integrated signals $\{R_{k,t}\}_{t=1}^n$ by
\begin{equation}\label{eq:rt}
R_{k,t} := \sum_{j=1}^t X_{k,j}, \quad k \in\{1,2\}, \  \ t\in\{1,\dots, n\}.
\end{equation}
Fig. \ref{f:xtandrt} presents the time series plot of a sample $\{X_t\}_{t=1}^{50}$ from a stationary time series and the corresponding integrated time series $\{R_t\}_{t=1}^{50}$.
\begin{figure}[!ht]
    \centering
 \includegraphics[width=\textwidth]{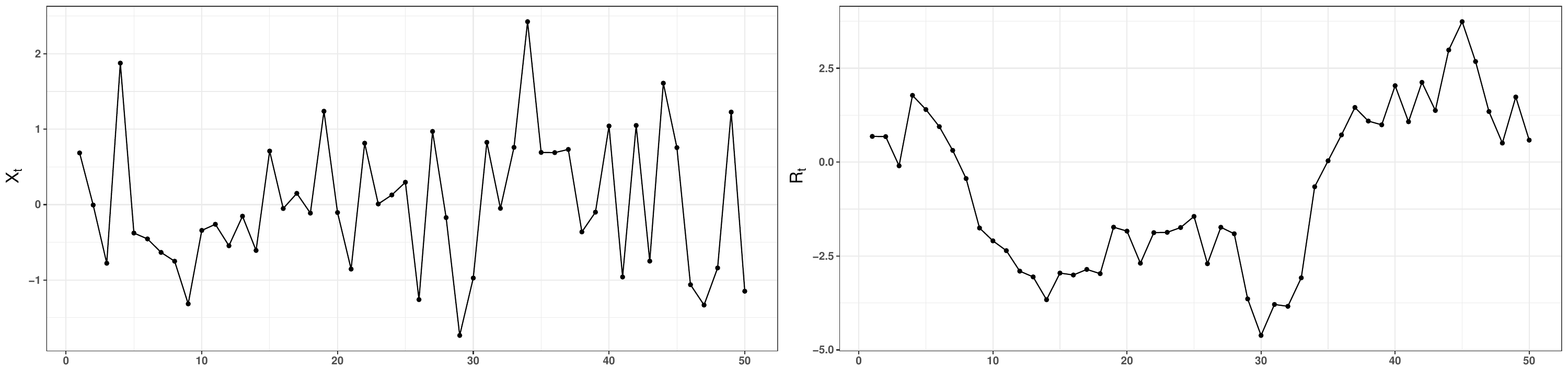}
  \caption{A simulated sample  $\{X_t\}_{t=1}^{50}$ from a stationary time series and the corresponding integrated time series $\{R_t\}_{t=1}^{50}$.}\label{f:xtandrt}
\end{figure}

Let $J_{\ell}$ be the $\ell\times \ell$ matrix whose  $(r,s)$th element is given by
\begin{equation}\label{eq:J}
  [J_\ell]_{r,s} = \left\{ \begin{array}{ll}
                      1, & \mbox{if} \quad 1 \leq s \leq r \leq \ell,\\
                      0,& \mbox{if} \quad 1 \leq r < s \leq \ell.
                    \end{array}
                    \right.
\end{equation}
It follows that, for $0 < m < n$,
\begin{equation}\label{eq:RRi}
\bs{R}_{k,n}^{\lr{1}} = J_n\bs{X}_{k,n}^{\lr{1}}, \qquad  \bs{R}_{k,m+i}^{\lr{i}}  =  J_{m+i}^{\lr{i}}\bs{X}_{k,m+i}^{\lr{1}}, \quad i \in\{1,\dots,n-m\}.
\end{equation}
The set $\big\{\bs{R}_{k,m+i}^{\lr{i}} \big\}_{i=1}^{n-m}$ is a sequence of $n-m$ overlapping boxes each containing $m+1$ {values from the integrated signal}, starting at $i$ and ending at $m+i$.

\begin{remark}
Notice that, upon replacing the right-hand-side equality in \eqref{eq:RRi} by
\[
\bs{R}_{k,(m+1)i}^{\lr{(m+1)(i-1)+1}} = J_{(m+1)i}^{\lr{(m+1)(i-1)+1}}\bs{X}_{k,(m+1)i}^{\lr{1}}, \quad i \in\big\{ 1, \ldots, \lfloor n/(m+1) \rfloor\big\},
\]
the corresponding set $\Big\{\bs{R}_{k,(m+1)i}^{\lr{(m+1)(i-1)+1}} \Big\}_{i=1}^{\lfloor n/(m+1) \rfloor}$ is a sequence of $\lfloor n/(m+1) \rfloor$ non-overlapping boxes each containing $m+1$ {values from the integrated signal}, starting at $(m+1)(i-1)+1$ and ending at $(m+1)i$. All calculations and results that follow can be obtained analogously for this case.
\end{remark}

Now, for each $k \in\{1,2\} $ and $i \in\{1,\dots,n-m\}$,  let $\tilde{\bs{R}}_{k,i}$ be the vector with the ordinates $\tilde R_{k,t}(i)$, $i \leq t \leq m+i$,  of a polynomial least-squares fit associated to the $i$th box $\bs{R}_{k,m+i}^{\lr{i}}$ and  $\bs{\Eps}_{k,i}$ be the vector with the corresponding error terms $\Eps_{k,t}(i)$, $i \leq t \leq m+i$, that is,
 \begin{align}
\tilde{\bs{R}}_{k,i} & = P_{m+1}\bs{R}_{k,m+i}^{\lr{i}} = \big(\tilde R_{k,i}(i), \dots, \tilde R_{k,m+i}(i)\big)^\top  ,\nonumber\\
\bs{\Eps}_{k,i} & = \bs{R}_{k,m+i}^{\lr{i}}   - \tilde{\bs{R}}_{k,i}  =  Q_{m+1}\bs{R}_{k,m+i}^{\lr{i}} = \big(\Eps_{k,i}(i), \dots, \Eps_{k,m+i}(i)\big)^\top  , \label{eq:epsk}
\end{align}
with
\begin{equation}
D_{m+1}^\top  := \left(
            \begin{array}{cccc}
              1 & 1 & \dots & 1 \\
              1 & 2 & \dots & m+1 \\
              \vdots&\vdots&\ddots&\vdots\\
              1^{\nu+1} & 2^{\nu+1} &\cdots &(m+1)^{\nu+1}
            \end{array}
          \right)\!, \quad \begin{array}{l} \\ \ P_{m+1} := D_{m+1}(D_{m+1}^\top  D_{m+1})^{-1}D_{m+1}^\top  ,\\[.1cm] \ Q_{m+1} := I_{m+1} - P_{m+1} \end{array}\label{eq:XPQ}
\end{equation}
and $\nu \in \N$. The dependence on $\nu$ in the matrices $D_{m+1}$, $P_{m+1}$ and $Q_{m+1}$ and the dependence on $m$ in the vectors $\tilde{\bs{R}}_{k,i}$ and $\bs{\Eps}_{k,i}$  was suppressed for simplicity. The error term $\bs{\Eps}_{k,i}$  is often called the ``detrended walk''.  Notice that, since we are considering overlapping boxes, the index $i$ in the notation $\tilde R_{k,t}(i)$ and $\Eps_{k,t}(i)$ is necessary to {indicate the boxes to which the values are associated}. {Fig. \ref{f:xtandrtbox} illustrates the local polynomial fit and, incidentally, the use of the notation, in the case of overlapping and non-overlapping boxes.}

\begin{figure}[!ht]
    \centering
 \includegraphics[width=\textwidth]{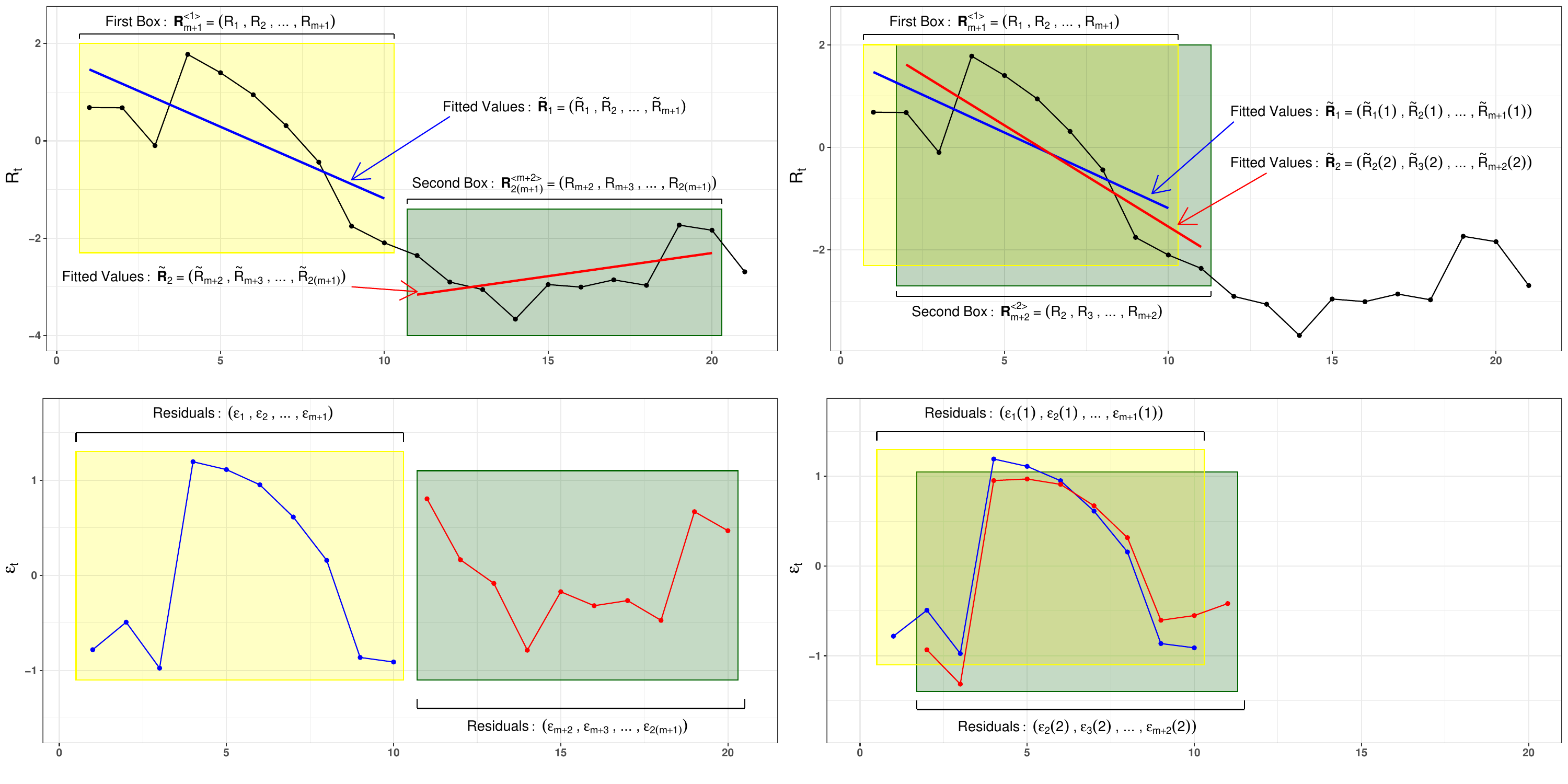}
  \caption{Visual representation of the integrated signal obtained from a stationary process and the corresponding local linear fit (top panels)  and residuals (bottom panels) in two different scenarios.  The left panels correspond to  non-overlapping boxes of size 10 $(m=9)$ while the right panels consider overlapping boxes of the same size.}\label{f:xtandrtbox}
\end{figure}

For $0 < m < n$ and $i \in\{ 1,\dots,n-m\}$, let $f^2_{k,DFA}(m,i)$ be the sample variance of the residuals $\big\{\Eps_{k,t}(i)\big\}_{t = i}^{m+i}$, for $k\in\{1,2\}$,  and  $f_{DCCA}(m,i)$ be the sample covariance between the residuals $\big\{\Eps_{1,t}(i)\big\}_{t = i}^{m+i}$  and $\big\{\Eps_{2,t}(i)\big\}_{t = i}^{m+i}$, corresponding to the $i$th box, that is,
\begin{equation}\label{eq:dfak_i}
f^2_{k,DFA}(m,i)  := \frac{1}{m}\sum_{t = i}^{m+i}\big(R_{k,t} - \tilde R_{k,t}(i)\big)^2 = \frac{1}{m} \bs{\Eps}_{k,i}^\top  \bs{\Eps}_{k,i}, \quad k \in\{ 1,2\},
\end{equation}
and
\begin{equation}\label{eq:dcca_i}
f_{DCCA}(m,i) := \frac{1}{m}\sum_{t = i}^{m+i}\big(R_{1,t} - \tilde R_{1,t}(i)\big)\big(R_{2,t} - \tilde R_{2,t}(i)\big)
= \frac{1}{m} \bs{\Eps}_{1,i}^\top  \bs{\Eps}_{2,i}.
\end{equation}
The {detrended variance} $F^2_{k,DFA}$, $k=1,2$, the {detrended covariance} $F_{DCCA}$ and the {detrended correlation coefficient} $\rho_{DCCA}$ are defined, respectively, by
\begin{equation}
F^2_{k,DFA}(m) := \frac{1}{n-m}\sum_{i = 1}^{n-m} f^2_{k,DFA}(m,i), \quad
F_{DCCA}(m) := \frac{1}{n-m}\sum_{i = 1}^{n-m} f_{DCCA}(m,i),\label{eq:Fs}
\end{equation}
and
\begin{equation}\label{eq:rho_dcca}
\rho_{DCCA}(m) := \frac{F_{DCCA}(m)}{\sqrt{F_{1,DFA}^2(m)} \sqrt{F_{2,DFA}^2(m)}}.
\end{equation}

{Notice that, from \eqref{eq:dfak_i}, \eqref{eq:dcca_i}, and upon invoking the Cauchy-Schwarz inequality twice, we obtain
 \begin{align*}
\bigg|\sum_{i = 1}^{n-m} f_{DCCA}(m,i)\bigg| & \leq \sum_{i = 1}^{n-m} f_{1,DFA}(m,i)f_{2,DFA}(m,i) \leq\sqrt{\sum_{i = 1}^{n-m} f_{1,DFA}^2(m,i)}\sqrt{\sum_{i = 1}^{n-m}f_{2,DFA}^2(m,i) }.
\end{align*}
Hence, just as the classical Pearson correlation coefficient, the DCCA also satisfy $|\rho_{DCCA}(m)|\leq 1$. }

\begin{remark}
  In the literature, the expression in \eqref{eq:dcca_i} and the detrended covariance are often {denoted respectively by} $f^2_{DCCA}(m,i)$ and $F^2_{DCCA}(m)$ instead of $f_{DCCA}(m,i)$ and $F_{DCCA}(m)$. Here we consider the second notation as it is more coherent with traditional notation of variance, covariance and correlation of random variables, namely, $\var(X) = \sigma^2_X$, $\cov(X,Y) = \sigma_{X,Y}$ and $\corr(X,Y) = \rho_{X,Y}$. Moreover, the notation $f^2_{DCCA}(m,i)$ and $F^2_{DCCA}(m)$ is somewhat misleading and can induce the reader to draw the wrong conclusion that the detrended covariance is always positive.
\end{remark}

As mentioned before, the theory presented in \cite{BK2008} for the DFA and \cite{Blythe2013} and \cite{Blythe2016} for the DCCA is developed under a (stationary) fractional Gaussian noise (or fractional Brownian motion) plus a polynomial trend assumption, which results in a {trend-stationary} process. The next result {shows} that the DFA and DCCA are always invariant under polynomial trend, {regardless the underlying distribution}, so that, in a trend-stationary context, {it suffices to} focus on the underlying stationary process. For ease of presentation, {we defer the proof of all results to the Appendix}.

\begin{prop}[Invariance to polynomial trend]\label{inv}
Let $\{X_{1,t}\}_{t\in\Z}$ and $\{X_{2,t}\}_{t\in\Z}$ be any two stochastic process and let $p_{1}$ and $p_{2}$ denote two polynomial of degree $\nu_1$ and $\nu_2$, respectively. Let $Y_{k,t}:=X_{k,t}+p_{k}(t)$, $k=1,2$, and $\nu = \max\{\nu_1, \nu_2\}$. Then the DFA and DCCA of $\{X_{k,t}\}_{t\in\Z}$ and $\{Y_{k,t}\}_{t\in\Z}$ are the same.
\end{prop}

The invariance to polynomial trend property of the DFA and DCCA {allows extending the results for a} stationary process to non-stationary process that can be written as a polynomial trend plus a stationary signal without loss of generality.

\section{Stationarity results and the theoretical counterpart of $\rho_{DCCA}$\label{sec:3}}

In this section, we derive several results for the DFA and DCCA and also present the theoretical counterpart of $\rho_{DCCA}$. We shall enunciate the results for (jointly) stationary processes only. {Because of Proposition \ref{inv}, the results in this section hold unaltered for trend-stationary processes, except stated otherwise}.

{We start by noticing} that  $P_{m+1}$ (the projection matrix) and $Q_{m+1}$ are bisymmetric,  hermitian and idempotent matrices. Moreover,
\begin{equation}\label{eq:eeps}
\E(\bs{\Eps}_{k,i})  = Q_{m+1}\E(\bs{R}_{k,m+i}^{\lr{i}} ) =  0_{m+1} \qquad \mbox{ if, and only if, } \qquad \E(\bs{R}_{k,m+i}^{\lr{i}} ) = D_{m+1}\bs\beta_i,
\end{equation}
for some $\bs\beta_i \in\R^{m+1}$, or, in other words, if $\E(\bs{R}_{k,m+i}^{\lr{i}} )$ is a polynomial trend of degree at most $\nu+1$. From \eqref{eq:rt}, $ \E(R_{k,t}) =  \sum_{j=1}^t \E(X_{k,j})$,  $k = 1,2$ and  $ t\geq 1$, so that a sufficient condition for \eqref{eq:eeps} to hold is that $\E(X_{k,j})$ does not depend on $j$. {Hence, if $\{X_{k,t}\}_{t\in\Z}$ is stationary, we have}
\[
Q_{m+1} J_{m+1+h}^{\lr{1+h}}\big[\bs{X}_{k,m+1+h}^{\lr{1}} - \E(\bs{X}_{k,m+1+h}^{\lr{1}})\big]= Q_{m+1} J_{m+1+h}^{\lr{1+h}}\bs{X}_{k,m+1+h}^{\lr{1}}, \quad h \geq 0,
\]
{so that, without loss of generality, we can assume that $\E(X_{k,t}) = 0$}.

Notice that, for a fixed $i$,  $\{\Eps_{k,t}(i)\}_{t=i}^{m+i}$ is not an identically distributed sequence. However, Lemma \ref{lem:eps} provides a sufficient condition for the joint stationarity of $\{\bs{\Eps}_{1,i}\}_{i=1}^{n-m}$ and $\{\bs{\Eps}_{2,i}\}_{i=1}^{n-m}$.

\begin{lemma}\label{lem:eps}
{If $\{X_{1,t}\}_{t\in\Z}$ and $\{X_{2,t}\}_{t\in\Z}$ are two jointly strictly stationary processes, then so are} $\{\bs{\Eps}_{1,i}\}_{i=1}^{n-m}$ and $\{\bs{\Eps}_{2,i}\}_{i=1}^{n-m}$.
\end{lemma}

As a consequence of Lemma \ref{lem:eps} we have Corollary \ref{lem:sta}, {which} shows that, if $\{X_{1,t}\}_{t\in\Z}$ and  $\{X_{2,t}\}_{t\in\Z}$ are jointly strictly stationary, then the distribution of $f_{k,DFA}^2(m,i)$ and $f_{DCCA}(m,i)$ do not depend on $i$. This result generalizes lemma 2.2 in \citet{BK2008}, where the authors show that $\{f_{k,DFA}^2(m,i)\}_{i=1}^{\lfloor n/m\rfloor }$, obtained by considering non-overlapping boxes of size $\lfloor n/m \rfloor$, and also lemma 1.3 in \cite{Blythe2016}, where the authors consider only the case where the underlying process is long-range dependent.

\begin{coro}\label{lem:sta}
{If} $\{X_{1,t}\}_{t\in\Z}$ and $\{X_{2,t}\}_{t\in\Z}$ are two jointly strictly stationary processes, {then} both processes, $\{f_{k,DFA}^2(m,i)\}_{i=1}^{n-m}$ and $\{f_{DCCA}(m,i)\}_{i=1}^{n-m}$, are strictly stationary.
\end{coro}

As mentioned in the introduction, {the DCCA and the DFA} are usually defined in terms of a sample from {a given} underlying process, as in \eqref{eq:Fs} and \eqref{eq:rho_dcca}. {Thus, they} can be viewed as an estimator of some quantity. {It is interesting to notice, however, that the literature does not mention the DCCA's theoretical counterpart}. To fill in this gap, observe that  as a consequence of Corollary \ref{lem:sta} and \eqref{eq:Fs},
 \[
\E\big(F^2_{k,DFA}(m)\big) = \frac{1}{n-m}\sum_{i = 1}^{n-m}\E\big(f^2_{k,DFA}(m,i)\big) = \E\big(f^2_{k,DFA}(m,1)\big) = \frac{1}{m}\E\big(\bs\Eps_{k,1}^\top  \bs\Eps_{k,1}\big)
\]
and
\[
\E\big(F_{DCCA}(m)\big) = \frac{1}{n-m}\sum_{i = 1}^{n-m} \E\big(f_{DCCA}(m,i)\big) = \E\big(f_{DCCA}(m,1)\big) = \frac{1}{m}\E\big(\bs\Eps_{k,1}^\top  \bs\Eps_{k,2}\big).
\]
Hence, the theoretical counterpart of $\rho_{DCCA}(m)$  is given by
\begin{equation}\label{eq:thcount}
\rho_{\bs\Eps}(m)  := \frac{\E\big(F_{DCCA}(m)\big)}{\sqrt{\E\big(F^2_{1,DFA}(m)\big)}\sqrt{\E\big(F^2_{2,DFA}(m)\big) }}
= \frac{\sum_{t=i}^{m+i}\cov(\Eps_{1,t}(i),\Eps_{2,t}(i))}{\sqrt{\sum_{t=i}^{m+i}\var(\Eps_{1,t}(i)) }\sqrt{\sum_{t=i}^{m+i}\var(\Eps_{2,t}(i)) }},
\end{equation}
 $0 < m<n,$  $1 \leq i \leq n-m$,  where  $\bs{\Eps}_{k,i} = \big(\Eps_{k,i}(i), \dots, \Eps_{k,m+i}(i)\big)^\top $ is defined by \eqref{eq:epsk}, $k\in\{1,2\}$.

A closer look into \eqref{eq:thcount} suggests why the theoretical counterpart of the $\rho_{DCCA}$ is never mentioned in the literature and also why applications only focus on its decay: there is no simple interpretation for \eqref{eq:thcount}. Observe that \eqref{eq:thcount} can be rewritten as the average covariance divided by the square root of the average variances corresponding to the processes $\big\{\Eps_{1,t}(i)\big\}_{t = i}^{m+i}$  and $\big\{\Eps_{2,t}(i)\big\}_{t = i}^{m+i}$, which are the residuals of a local polynomial fit applied to the $i$th window associated to the integrated processes $\{R_{1,t}\}_{t=1}^n$ and $\{R_{2,t}\}_{t=1}^n$. Hence, it is clear that $\rho_{\bs\Eps}(m)$ is not a direct measure of the cross-correlation between the original processes. It is also obvious that $\rho_{DCCA}(m)$ is a biased estimator for $\rho_{\bs\Eps}(m)$, for fixed $n$. However, {as} it will be shown in the sequel, { $\rho_{\bs\Eps}(m)$ is consistent under some mild conditions.}

Another consequence of Corollary \ref{lem:sta} is that, if  $f_{k,DFA}^2(m,i)$ and $f_{DCCA}(m,i)$ have finite variance, then
\begin{equation}\label{eq:c1}
\gamma_{k,\mbox{\tiny DFA}}(h) :=  \cov \big(f^2_{k,DFA}(m,1),f_{k,DFA}^2(m,1+h)\big) = \cov \big(f^2_{k,DFA}(m,i),f_{k,DFA}^2(m,i+h)\big),
\end{equation}
\begin{equation}\label{eq:c2}
\gamma_{\mbox{\tiny DCCA}}(h) := \cov \big(f_{DCCA}(m,1),f_{DCCA}(m,1+h)\big) = \cov \big(f_{DCCA}(m,i),f_{DCCA}(m,i+h)\big), \end{equation}
for all $i \geq 1$ and $h \geq 0$ for which the last terms in \eqref{eq:c1} and \eqref{eq:c2} make sense.  Closed expressions for \eqref{eq:c1} and \eqref{eq:c2} shall be derived in the sequel, but first we need to introduce some notation.

For any $k_1, k_2 = 1,2$, $0 < m < n$ and $0 \leq  h_1, h_2 < n-m$,  let
\[
   \Gamma_{k_1,k_2}^{h_1,h_2}  :=  \cov\Big(\bs{X}_{k_1,m+1+h_1}^{\lr{1}}, \bs{X}_{k_2,m+1+h_2}^{\lr{1}}\Big), \qquad
    \Sigma_{k_1,k_2}^{h_1, h_2} := \cov\Big(\bs{R}_{k_1,m+1+h_1}^{\lr{h_1 + 1}},\bs{R}_{k_2,m+1+h_2}^{\lr{h_2+1}}\Big),
\]
and observe that, from \eqref{eq:RRi},
\begin{equation}\label{eq:SG}
\Sigma_{k_1,k_2}^{h_1, h_2}   = J_{m+1+h_1}^{\lr{h_1}}\Gamma_{k_1,k_2}^{h_1,h_2}\Big[J_{m+1+h_2}^{\lr{h_2}}\Big]^\top
\end{equation}
where $J_{m+1+h}$ is defined by \eqref{eq:J}, for $h \geq 0$. Let $\ka_{k_1, k_2}(p,r,q,s)$ denotes the joint cumulant of $X_{k_1,p}, X_{k_1,r},X_{k_2,q}, X_{k_2,s}$ and let $\K_{k_1,k_2}(h)$ be the  $[(m+1)(m+1+h)] \times [(m+1)(m+1+h)]$ block matrix, for which the  $(r,s)$th element in the $(p,q)$th block is given by
\begin{equation}\label{eq:cumu}
\Big[\big[\K_{k_1,k_2}(h)\big]^{p,q}\Big]_{r,s}  := \ka_{k_1, k_2}(p,r,q,s), \quad 1\leq p,q \leq m+1, \quad 1\leq r,s\leq m+1+h.
\end{equation}
For sake of simplicity, for any $h,h_1,h_2 \geq 0$  and $k,k_1,k_2 \in\{ 1,2\}$, define
\begin{equation}\label{eq:Gammak}
  \Gamma_k^{h_1,h_2} :=  \Gamma_{k,k}^{h_1,h_2}, \quad\Gamma_k :=  \Gamma_{k,k}^{0,0}, \quad  \Sigma_k^{h_1,h_2} := \Sigma_{k,k}^{h_1,h_2}, \quad  \Sigma_k := \Sigma_{k,k}^{0,0},
\end{equation}
\begin{equation}\label{eq:Gamma12}
  \Gamma_{1,2} :=  \Gamma_{1,2}^{0,0}, \quad \Sigma_{1,2} := \Sigma_{1,2}^{0,0}, \quad  \K_{k}(h) := \K_{k,k}(h), \quad \K_{k} := \K_{k}(0) \aand  \K_{k_1,k_2} := \K_{k_1,k_2}(0).
  \end{equation}
Moreover, let $K_{m+1} = K_{m+1}(0) := J_{m+1}^\top Q_{m+1} J_{m+1}$ and observe that,  for all $h > 0$,
\begin{equation}
J_{m+1+h}^{\lr{h+1}} = [1_{m+1,h} \ J_{m+1}] \implies
 K_{m+1}(h):= \Big[J_{m+1+h}^{\lr{h+1}}\Big]^\top Q_{m+1}J_{m+1+h}^{\lr{h+1}} =
\begin{bmatrix}
0_{h,h} & 0_{h,m+1}\\
0_{m+1,h} & K_{m+1}\\
 \end{bmatrix}. \label{eq:Km}
 \end{equation}
 Also, let $K_{m+1}^\otimes = K_{m+1}^\otimes(0) := K_{m+1}\otimes K_{m+1}$ and,  for $h > 0$,
\begin{equation}
  K_{m+1}^\otimes(h)   := \Big[J_{m+1}\otimes J_{m+1+h}^{\lr{h+1}}\Big]^\top \big(Q_{m+1}\otimes Q_{m+1}\big)\Big[J_{m+1}\otimes J_{m+1+h}^{\lr{h+1}}\Big]   = K_{m+1} \otimes K_{m+1}(h).  \label{eq:Kmo}
\end{equation}

Theorem \ref{thm:expec} presents closed form expressions for the expectation, variance and covariance function related to the processes $\{f_{k,DFA}^2(m,i)\}_{i=1}^{n-m}$ and $\{f_{DCCA}(m,i)\}_{i=1}^{n-m}$, under joint stationarity and finite fourth moment assumptions for $\{X_{1,t}\}_{t\in\Z}$ and $\{X_{2,t}\}_{t\in\Z}$. This result is a generalization of the results presented in \citet{BK2008} for the DFA and in \cite{Blythe2016} for the DCCA, where the authors consider non-overlapping windows and a fractional Gaussian noise (plus a polynomial trend) as the underlying process. Moreover, while the expressions derived in \cite{Blythe2016} are presented in terms of the covariance matrices related to the integrated process $\{(R_{1,t}, R_{2,t})\}_{t=1}^n$, the results given in Theorem \ref{thm:expec} are written in terms of the covariance matrices related to the original process $\{(X_{1,t}, X_{2,t})\}_{t \in \Z}$, which is often more useful.

\begin{thm}\label{thm:expec}
Let $\{X_{1,t}\}_{t\in\Z}$ and $\{X_{2,t}\}_{t\in\Z}$ be two jointly  strictly stationary stochastic processes with $\E(|X_{k,t}|^4) < \infty$, $k=1,2$.
Then, for all $0 < m < n$, $1 \leq  i \leq n-m$, $0  \leq h < n-m$ and $k \in\{ 1,2\}$,
\begin{align}
 & \E\big(f^2_{k,DFA}(m,i)\big)=  \frac{1}{m}\trace\big(K_{m+1}\Gamma_{k}\big), \label{eq:Edfai}\\
& \gamma_{k,\mbox{\tiny DFA}}(0) =  \frac{1}{m^2}\Big[\trace\big(K_{m+1}^\otimes\K_{k}\big) + 2\trace\big(K_{m+1}\Gamma_{k}K_{m+1}\Gamma_{k}\big)\Big], \label{eq:Vardfai}\\
 &\gamma_{k,\mbox{\tiny DFA}}(h)  =\frac{1}{m^2}\Big[\trace\big(K_{m+1}^\otimes(h)\K_{k}(h)\big) + 2\trace\big(K_{m+1}\Gamma_{k}^{0,h}K_{m+1}(h)\Gamma_{k}^{h,0}\big)\Big],\label{eq:Covdfai}
\end{align}
and
\begin{align}
&\E\big(f_{DCCA}(m,i)\big) = \frac{1}{m}\trace\big(K_{m+1}\Gamma_{1,2}\big), \label{eq:Edccai}\\
 &\gamma_{\mbox{\tiny DCCA}}(0) =  \frac{1}{m^2}\Big[\trace\big(K_{m+1}^\otimes\K_{1,2}\big) + \trace\big(K_{m+1}\Gamma_{1}K_{m+1}\Gamma_{2}\big) + \trace\big(K_{m+1}\Gamma_{1,2}K_{m+1}\Gamma_{1,2}\big)\Big], \label{eq:Vardccai}\\
&\gamma_{\mbox{\tiny DCCA}}(h) =  \frac{1}{m^2}\Big[\trace\big(K_{m+1}^\otimes\K_{1,2}(h)\big) + \trace\big(K_{m+1}\Gamma_{1}^{0,h}K_{m+1}(h)\Gamma_{2}^{h,0}\big) + \trace\big(K_{m+1}\Gamma_{1,2}^{0,h}K_{m+1}(h)\Gamma_{1,2}^{h,0}\big)\Big], \label{eq:Covdccai}
\end{align}
with $\Gamma_{k}$, $\Gamma_{k}^{h_1,h_2}$,  $\Gamma_{1,2}$, $\Gamma_{1,2}^{h_1,h_2}$,  $\K_{k}$,  $\K_{k}(h)$, $\K_{1,2}$,  $\K_{1,2}(h)$,   $K_{m+1}$,  $K_{m+1}^\otimes$, $K_{m+1}^\otimes(h)$ defined in \eqref{eq:cumu} - \eqref{eq:Kmo}, $h_1,h_2 \in\{0,h\}$.
\end{thm}

\begin{remark}
For non-overlapping windows, we define
 \[
 \Sigma_{k_1,k_2}^{h_1, h_2} := \cov\Big(\bs{R}_{k_1,(m+1)(h_1 + 1)}^{\lr{(m+1)h_1 + 1}},\bs{R}_{k_2,(m+1)(h_2 + 1)}^{\lr{(m+1)h_2 + 1}}\Big), \quad h_1, h_2 \geq 0,
 \]
 so that
 \[
 \Sigma_{k_1,k_2}^{h_1, h_2}   = J_{(m+1)(h_1+1)}^{\lr{(m+1)h_1 + 1}}\Gamma_{k_1,k_2}^{h_1,h_2}\Big[J_{(m+1)(h_2+1)}^{\lr{(m+1)h_2 + 1}}\Big]^\top , \quad \Gamma_{k_1,k_2}^{h_1,h_2}  :=  \cov\Big(\bs{X}_{k_1,(m+1)(h_1+1)}^{\lr{1}}, \bs{X}_{k_2,(m+1)(h_2+1)}^{\lr{1}}\Big).
 \]
Also, in \eqref{eq:cumu} make $1\leq r,s\leq (m+1)(h+1)$ and in \eqref{eq:Km} and \eqref{eq:Kmo} replace $J_{m+1+h}^{\lr{h+1}}$ with $J_{(m+1)(h+1)}^{\lr{(m+1)h+1}}$.  Theorem \ref{thm:expec} remains unchanged.
\end{remark}

As mentioned before, $\rho_{DCCA}(m)$ is a biased estimator for \eqref{eq:thcount}. However, using the results derived in Theorem \ref{thm:expec},  Theorem \ref{thm:convDCCA} provides sufficient conditions for consistence and almost sure convergence of $\rho_{DCCA}(m)$.  As a direct consequence of this theorem, it is showed that $\rho_{DCCA}(m)$ is asymptotically unbiased.

\begin{thm}\label{thm:convDCCA}
Let $\{X_{1,t}\}_{t\in\Z}$ and $\{X_{2,t}\}_{t\in\Z}$ be two jointly stationary processes.
If $\gamma_{k,\mbox{\tiny DFA}}(h) \to 0$ and $\gamma_{\mbox{\tiny DCCA}}(h) \to 0$, as $h\to \infty$,
then
  \[
F_{k, DFA}^2(m)  \overset{P}{\To}  \E\big(f_{k, DFA}^2(m,1) \big) = \frac{1}{m}\trace\big(K_{m+1}\Gamma_{k}\big), \as n\to \infty,
\]
and
 \[
F_{ DCCA}(m)  \overset{P}{\To} \E\big(f_{DCCA}(m,1) \big) = \frac{1}{m}\trace\big(K_{m+1}\Gamma_{1,2}\big), \as n\to \infty.
\]
Moreover,
\[
\rho_{DCCA}(m)  \overset{P}{\To} \frac{ \trace\big(K_{m+1}\Gamma_{1,2}\big)}{\sqrt{  \trace\big(K_{m+1}\Gamma_{1}\big) \trace\big(K_{m+1}\Gamma_{2}\big)}} = \rho_{\bs\Eps}(m), \as n\to \infty,
\]
where  $\Gamma_{k}$, $\Gamma_{1,2}$ and $K_{m+1}$ are defined, respectively,  by  \eqref{eq:Gammak},   \eqref{eq:Gamma12} and  \eqref{eq:Km}.
Furthermore, if
\[
\sum_{h = 1}^\infty \frac{|\gamma_{k,DFA}(h)|}{h^{q_k}} < \infty, \qquad \sum_{h = 1}^\infty \frac{|\gamma_{DCCA}(h)|}{h^{q_{12}}} < \infty,
\]
for some $0 \leq q_k, q_{12} < 1$, then these convergences hold almost surely.
\end{thm}

{In order to derive the limit in probability of $F_{k, DFA}^2(m)$, $F_{ DCCA}(m)$ and $\rho_{DCCA}(m)$, as $n\to\infty$, Theorem \ref{thm:convDCCA} requires that $\gamma_{k,\mbox{\tiny DFA}}(h) \to 0$ and $\gamma_{\mbox{\tiny DCCA}}(h) \to 0$, as $h\to \infty$. These conditions are very mild, but can still be difficult to verify in some specific contexts.
Proposition \ref{lem:cond} presents sufficient conditions  for the hypothesis of Theorem \ref{thm:convDCCA} to hold. These conditions are related to the behavior of the cross-covariance and joint cumulants of the original time series, often much simpler to verify}.
\begin{prop}\label{lem:cond}
Suppose that, for $k_1,k_2\in\{ 1,2\}$ and any $p,q,\tau >0$ fixed,
\[
 \gamma_{k_1,k_2}(h) \to 0, \mbox{as } |h| \to \infty \aand \ka_{k_1,k_2}(p,h+\tau,p,h+q) \to 0, \as h \to \infty,
\]
where {$\gamma_{k_1,k_2}(h):=\cov(X_{k_1,t},X_{k_2,t+h})$}. Then $\gamma_{k,\mbox{\tiny DFA}}(h)  \to 0$ and $\gamma_{\mbox{\tiny DCCA}}(h) \to 0$,  as  $h \to \infty$.
\end{prop}


\section{Specializations } \label{special}

The results presented in Theorem \ref{thm:convDCCA} are general ones, holding for jointly stationary processes with some mild assumptions. In this section we examine the consequences of Theorem \ref{thm:convDCCA} in some specific contexts, focusing in asymptotic results when both, the sample size $n$ and the window size $m$ go to infinity.

We {start by studying} the asymptotic behavior of $F_{k, DFA}^2(m)$ and $F_{DCCA}(m)$, as $m\to \infty$, but first we need to introduce some notation. Recall that the degree of the polynomial fit in \eqref{eq:XPQ} is given by $\nu+1$. Let $M_j$ and  $M_j^\ast$, for all $0 \leq j \leq m$, be two matrices of size $(m+1)\times(m+1)$ for which the $(r,s)$th coefficients are given respectively by
\begin{equation}\label{eq:matrixMast}
[M_j]_{r,s} = \left\{
\begin{array}{cc}
1, & \mbox{if} \quad |r-s| = j,\\
0, & \mbox{otherwise},
\end{array}
\right. \qquad
[M_j^\ast]_{r,s} = \left\{
\begin{array}{cc}
1, & \mbox{if} \quad s-r = j,\\
0, & \mbox{otherwise}.
\end{array}
\right.
\end{equation}
Observe that the matrices  $\Gamma_{k}$ and $\Gamma_{1,2}$ defined respectively  by  \eqref{eq:Gammak} and   \eqref{eq:Gamma12}, can be rewritten as
\begin{equation}
\Gamma_k = \sum_{h = 0}^{m}\gamma_k(h)M_h, \qquad\Gamma_{1,2} = \gamma_{1,2}(0)M_0^\ast + \sum_{h = 1}^{m}\gamma_{1,2}(h)M_h^\ast + \sum_{h = 1}^{m}\gamma_{1,2}(-h)\big[M_h^\ast\big]^\top .\nonumber
\end{equation}
Let
\begin{equation}\label{eq:alpha}
 \alpha_0^{(m)} := \trace \big(K_{m+1}\big), \quad  \alpha_j^{(m)} :=\trace (K_{m+1}M_j^\ast),\quad \beta_j^{(m)} := \alpha_j^{(m)}/\alpha_0^{(m)}, \quad 1\leq j\leq m,
\end{equation}
for $K_{m+1}$ defined by \eqref{eq:Km}, so that, for any $\nu\geq 0$ one can write
\begin{align}
  \trace(K_{m+1}\Gamma_k) &
                            = \alpha_0^{(m)}\bigg\{\gamma_k(0) + 2\sum_{h = 1}^{m-1}\beta_h^{(m)}\gamma_k(h) \bigg\},\label{eq:gammask}\\
  \trace(K_{m+1}\Gamma_{1,2}) 
                        &= \alpha_0^{(m)}\bigg\{\gamma_{1,2}(0) + \sum_{h = 1}^{m-1}\beta_h^{(m)}\Big[\gamma_{1,2}(-h) + \gamma_{1,2}(h)\Big]\bigg\}.\label{eq:gammas12}
\end{align}
These equations yield Lemma \ref{lem:alpham}. {In what follows, for two sequences of real numbers $\{a_n\}_{n=1}^\infty$ and $\{b_n\}_{n=1}^\infty$, we write $a_n\sim b_n$ if $a_n/b_n\to 1$, as $n\to\infty$.}

\begin{lemma}\label{lem:alpham}
Let  $\nu+1$ be the degree of the polynomial fit in \eqref{eq:XPQ} and $\alpha_j^{(m)}$, $0\leq j\leq m$, be defined by \eqref{eq:alpha}. Also, let $K_{m+1}$  and $M_j$ be the matrices defined by \eqref{eq:Km} and \eqref{eq:matrixMast}, respectively. Then
\begin{equation*}
  \trace (K_{m+1}M_0) = \alpha_0^{(m)}, \quad    \trace (K_{m+1}M_m) = 0, \quad \trace (K_{m+1}M_j) = 2\alpha_j^{(m)}, \quad 1 \leq j <m.
\end{equation*}
In particular, if $\nu=0$, then $ \alpha_j^{(m)}\sim \alpha_0^{(m)}$, as $m\to \infty$.
\end{lemma}

From the proof of Lemma \ref{lem:alpham} {(presented in the Appendix)}, when $\nu=0$ one concludes that
\[
0 < \beta_j^{(m)} \leq 1, \quad  \mbox{if} \quad  0\leq j <  j_0(m), \aand -1 \leq \beta_j^{(m)} \leq 0, \quad  \mbox{if} \quad
j_0(m) \leq j < m,
\]
where $j_0(m) = \big[\sqrt{105 m^2 + 210 m + 9} - 9(m +1)\big]/6$.  Also,  from Lemma \ref{lem:alpham},  for all $0 \leq j <m$,   $\beta_j^{(m)} \to 1$, as $m \to \infty$.
In Theorem \ref{thm:asymp},   \eqref{eq:gammask}, \eqref{eq:gammas12} and Lemma \ref{lem:alpham} are combined to obtain the asymptotic behavior of $\E\big(f^2_{k,DFA}(m,1)\big)$ and  $\E\big(f_{DCCA}(m,1)\big)$, as $m\to \infty$. Further notice that, under the hypothesis of Theorem \ref{thm:convDCCA},  the expressions derived also provide the asymptotic behavior of $F_{DCCA}(m)$, $F_{k,DFA}^2(m)$ and $\rho_{DCCA}(m)$:

\begin{thm}\label{thm:asymp}
Let $\{X_{1,t}\}_{t\in\Z}$ and $\{X_{2,t}\}_{t\in\Z}$ be two jointly  strictly stationary stochastic processes with $\E(|X_{k,t}|^4) < \infty$,  autocovariance $\gamma_{k}(\cdot)$, $k=1,2$,  and cross-covariance $\gamma_{1,2}(\cdot)$.  If  $\nu=0$,  $\sum_{h\in \Z}|\gamma_{k}(h)| < \infty$ and $\sum_{h\in \Z}|\gamma_{1,2}(h)| < \infty$, then
\begin{equation}\label{eq:exps}
\E\big(f^2_{k,DFA}(m,1)\big) \sim \frac{m}{15} \sum_{h\in\Z}\gamma_k(h), \qquad
\E\big(f_{DCCA}(m,1)\big) \sim \frac{m}{15} \sum_{h\in\Z}\gamma_{1,2}(h), \as m\to \infty.
\end{equation}
\end{thm}

\noindent From Theorem \ref{thm:asymp} we conclude that, under the conditions of Theorem \ref{thm:convDCCA} {and  $\nu=0$},
\begin{equation}\label{eq:rhoasymp}
\rho_{DCCA}(m)  \  \overset{P} \To  \  \frac{\sum\limits_{h\in \Z}\gamma_{1,2}(h) }{\sqrt{\sum\limits_{h\in \Z}\gamma_1(h)} \sqrt{\sum\limits_{h\in \Z}\gamma_2(h)} }, \as m,n \to \infty.
\end{equation}

{Fig. \ref{f:asym}  presents the behavior of $F_{DCCA}(m)$, $F_{k,DFA}^2(m)$ and $\rho_{DCCA}(m)$ calculated from a sample of size 50 obtained from a jointly stationary process. The plots display the results for overlapping (red) and non-overlapping (blue) windows, as well as the asymptotic relation presented in Theorem \ref{thm:asymp} (green).
}

\begin{figure}[h]
\centering
\includegraphics[width=\textwidth]{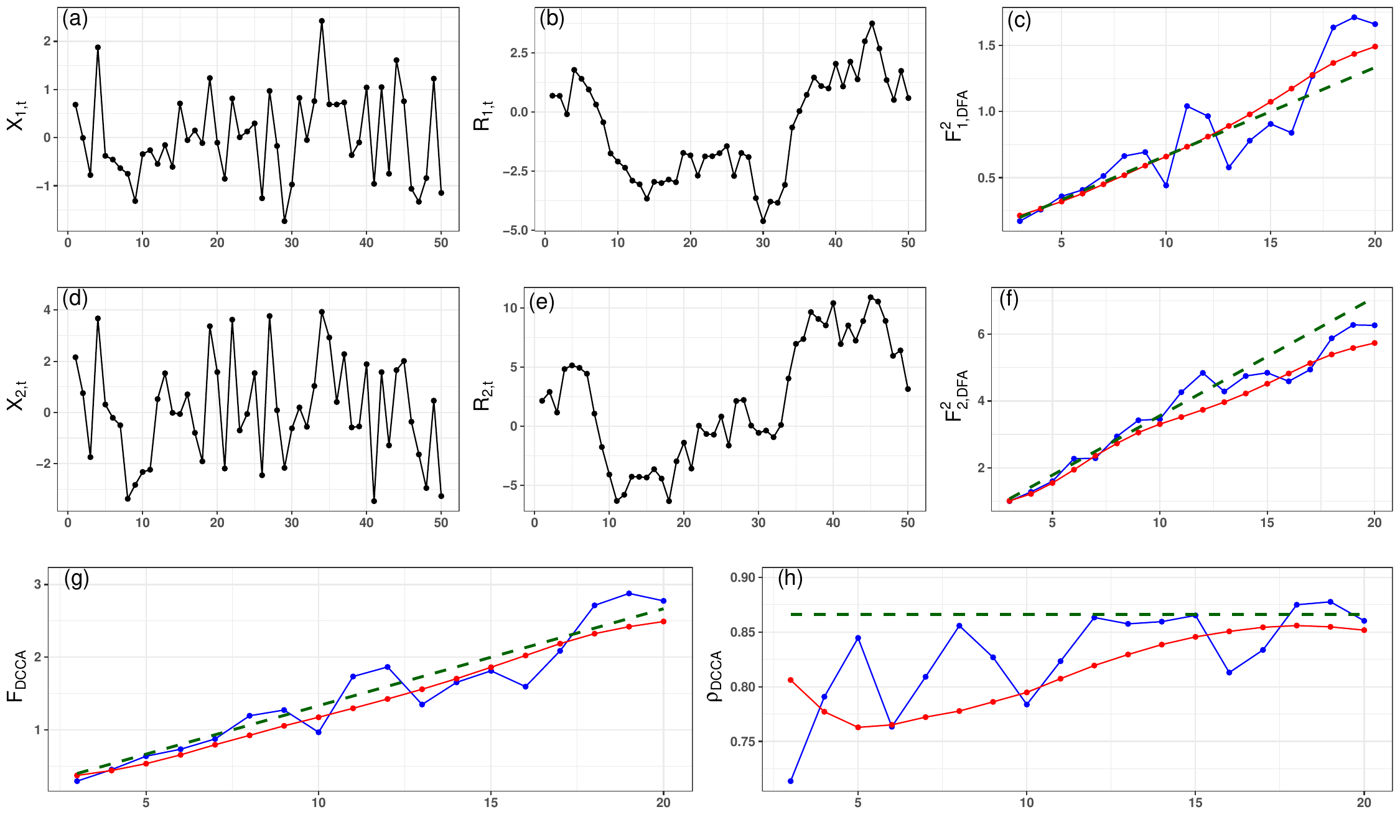}
\caption{
Consider jointly stationary processes $(X_{1,t}, X_{2,t})$ obtained as follows: $X_{2,t}= 2X_{1,t}+\eps_t$, where $X_{1,t}$ is i.i.d $\mathcal N(0,1)$ and $\eps_t$ is i.i.d $\mathcal U(-2,2)$. Plots of $X_{1,t}$ and $X_{2,t}$ are presented in (a) and (d); (b) and (e) present the respective associate integrated signal; (c) and (f) present the associated $F_{k,DFA}^2(m)$ as a function of $m$. $F_{DCCA}$ and $\rho_{DCCA}$ are presented in (g) and (h).
In each plot, the blue line represents the respective quantities calculated using non-overlapping windows, while the same quantities calculated using overlapping windows are presented in red.
The green line is the appropriate right hand side of \eqref{eq:exps} or \eqref{eq:rhoasymp}.
}\label{f:asym}

\end{figure}

Wold's decomposition theorem states that any zero mean nondeterministic weakly stationary  process $\{X_t\}_{t\in\Z}$ can be {uniquely} decomposed as $X_t = \sum_{j = 0}^\infty \psi_j u_{t-j} + d_t$, where  $\sum_{j = 0}^\infty \psi_j^2 < \infty$, $\{u_t\}_{t\in \Z}$ is  a white noise process and $\{d_t\}_{t\in \Z}$ is a deterministic process. Inspired by this result,  Corollary \ref{coro:linear} provides {the} limit in probability for the  $\rho_{DCCA}$ in the context of linear sequences with absolutely summable coefficients, described in the sequel.

Let  $\{X_{1,t}\}_{t\in\Z}$  and $\{X_{2,t}\}_{t\in\Z}$ be  two jointly stationary processes satisfying
\begin{equation}\label{eq:linear}
X_{k,t} = \sum_{j\in\Z} \psi_{k,j} \eps_{k,t-j}, \quad t \in \Z, \quad \mbox{with} \quad \sum_{j\in\Z} |\psi_{k,j}| < \infty, \quad k\in\{1,2\}.
\end{equation}
where  $\{\eps_{k,t}\}_{t\in \Z}$, $k\in\{1,2\}$, are white noise processes with zero mean,  $\var(\eps_{k,t}) = \tau_k^2$, $k\in\{1,2\}$, and $\cov(\eps_{1,r}, \eps_{2,s}) = \tau_{1,2}$, if $r=s$, and zero otherwise.
Structure  \eqref{eq:linear} can be used to describe a wide variety of scenarios. For instance,
\begin{enumerate}[label =  \arabic*), wide = 0.5em, itemsep = 0em, nosep, leftmargin = 2em]
\item if   $\tau_{1,2} = 0$, then $\{X_{k,t}\}_{t\in \Z}$, $k\in\{1,2\}$,  are stationary uncorrelated processes;

\item if  $\psi_{k,j} = 0$, for  all $j \neq  0$, and  $\tau_{1,2} \neq 0$,  then  $\{(X_{1,t}, X_{2,t})\}_{t\in \Z}$ is a  bivariate white noise;

\item if  $\psi_{k,j} = 0$, for  all $j < 0$ and $j > j_0$, for some $j_0 \in \N$,  $\{X_{k,t}\}_{t\in \Z}$ for $k\in\{1,2\}$ are  moving average  processes;

 \item if $\psi_{k,j} = 0$, for  all $j < 0$, then  $\{X_{k,t}\}_{t\in \Z}$ for $k\in\{1,2\}$ are causal processes.
\end{enumerate}

\noindent Observe that
\begin{equation}\label{eq:gammasX}
    \gamma_{k}(h) = \tau_k^2\sum_{j\in\Z}\psi_{k,j}\psi_{k, j+h}, \quad
   \gamma_{1,2}(h) = \tau_{1,2}\sum_{j\in\Z} \psi_{1,j}\psi_{2,j+h}, \quad \mbox{for all} \quad h \in Z,
\end{equation}
  so that $\gamma_{k}(h)  \to 0$ and $ \gamma_{1,2}(h)  \to 0$, as $|h| \to \infty$. If, in particular, $\{(\eps_{1,t}, \eps_{2,t})\}_{t\in \Z}$ is such that $\eps_{k_1,t} \ind \eps_{k_2,s}$, $t \neq s$ and $k_1, k_2 \in\{ 1,2\}$, (that is, the random variables are independent for $t \neq s$), with finite joint fourth moment, then
\begin{align*}
\cov(X_{k_1,p}X_{k_2,p}, X_{k_1,\tau + h}X_{k_2,q + h}) =  \var(\eps_{k_1,0}\eps_{k_2,0})\sum_{j \in \Z}\psi_{k_1,j}\psi_{k_1,j-p+\tau+h}\psi_{k_2,j}\psi_{k_2,j-p+q+h} \To 0,
\end{align*}
as $ h \to \infty$, for $k_1,k_2 \in\{ 1,2\}$, which is equivalent to
$\ka_{k_1,k_2}(p,\tau+h,p,q+h) \to 0$, as  $h \to \infty$. Moreover, if $\{(\eps_{1,t}, \eps_{2,t})\}_{t\in \Z}$ is a bivariate Gaussian process then
\[
\ka_{k_1,k_2}(p,\tau+h,p,q+h) = 0, \quad \forall \, p, h, \tau, q \in \Z.
\]

\begin{coro}\label{coro:linear}
Let  $\{X_{1,t}\}_{t\in\Z}$  and $\{X_{2,t}\}_{t\in\Z}$ be  two jointly stationary processes satisfying \eqref{eq:linear} and let $\nu=0$. Then, under the conditions of Theorem \ref{thm:convDCCA},
\[
\rho_{DCCA}(m)\overset{P}{\To}  \sign(\Psi_{1,2}) \frac{\tau_{1,2}}{\tau_1\tau_2} = \sign(\Psi_{1,2})\rho_{1,2}^\varepsilon,
\quad \as n,m\to \infty,
  \]
where $\Psi_{1,2} = \sum_{j\in \Z} \psi_{1,j}\sum_{\ell\in\Z}\psi_{2,\ell}$ and  $\rho_{1,2}^\varepsilon = \corr(\eps_{1,t}, \eps_{2,t})$.
\end{coro}

\subsection{Effect of $\nu$ on Theorem \ref{thm:convDCCA} }

In view of Proposition \ref{inv}, the results in Lemma \ref{lem:alpham}, Theorem \ref{thm:asymp}, and Corollary \ref{coro:linear} also hold for stationary process with linear trends, without any modification. Similar results can, in principle, be obtained for polynomial trends of higher degrees as long as one {can} carry on the delicate analysis in Lemma \ref{lem:alpham}. It entails a careful investigation of the general form of $Q_{m+1}$ in \eqref{eq:XPQ}, which depends on obtaining $(D_{m+1}^\top D_{m+1})^{-1}$ in closed form. {Even for $\nu=1$, this task is very complicated. Also, the analysis is $\nu$ by $\nu$ dependent in the sense that no general formula can be obtained.}

To visualize the effect of $\nu>0$ on the asymptotic behavior of the $F^2_{k,DFA}(m)$, $F_{DCCA}(m)$ and $\rho_{DCCA}(m)$, we perform a Monte Carlo Simulation study considering $\nu\in\{0,2,5\}$ and $1{,}000\leq m\leq1{,}100$ with increments of size 10. The data generating process in this example is
\[X_{1,t} = 0.6X_{1,t-1}+\eps_t, \quad \eps_t\simiid \mathcal N(0,1), \quad X_{2,t} = 0.6\eta_{t-1}+\eta_t, \quad \eta_t\simiid \mathcal N(0,1),\]
with $\eps_t$ and $\eta_t$ independent of each other. We perform 1,000 replications and the sample size was set to 2,000 in each case. The results are shown in Fig. \ref{f:nu025}.  The difference between the $F^2_{k,DFA}(m)$, $F_{DCCA}(m)$ and $\rho_{DCCA}(m)$ for $\nu=2$ and $\nu=5$ is shown in the rightmost column. From the plots, the similarity between the plots is remarkable. Just as in the case of $\nu=0$, the $F^2_{k,DFA}(m)$, $F_{DCCA}(m)$ and $\rho_{DCCA}(m)$ behaves as a linear function of $m$ for $\nu\in\{2,5\}$.

Observe that the degree of the local polynomial fitted on each window of the integrated noise is, by definition, $\nu+1$. It is a well-known fact from regression analysis that, increasing the degree of a polynomial fit tends to present diminishing returns on the quality of the approximation. So, it is expected that the values of $F^2_{k,DFA}(m)$, $F_{DCCA}(m)$ and $\rho_{DCCA}(m)$ tend to stabilize as $\nu$ increases.

\begin{figure}[h]
\centering
\includegraphics[width=\textwidth]{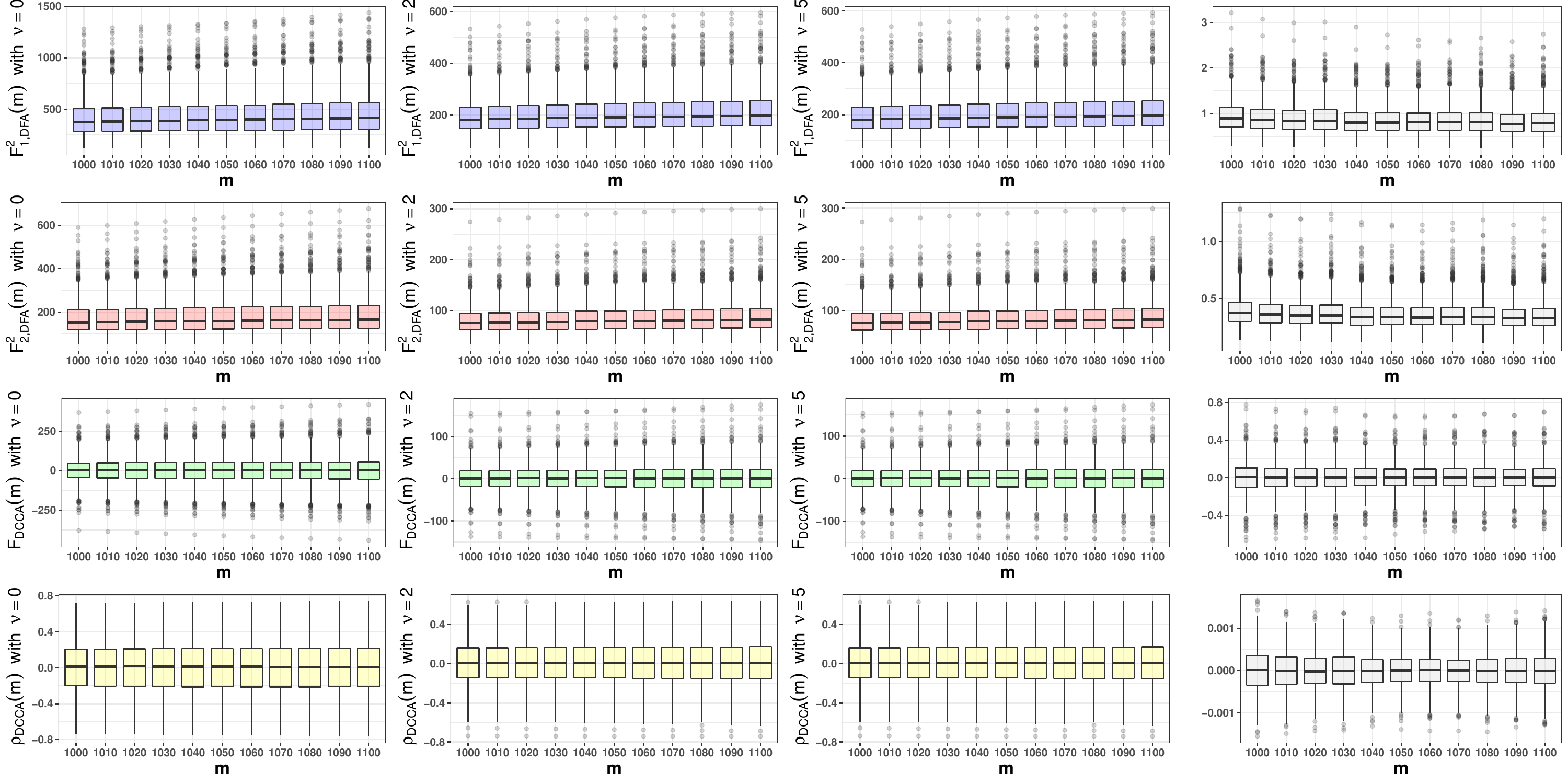}
\caption{Boxplots of the $F^2_{k,DFA}(m)$, $F_{DCCA}(m)$ and $\rho_{DCCA}(m)$ for $\nu\in\{0,2,5\}$. The rightmost panel shows the boxplots of the differences for $\nu=2$ and $\nu=5$ in each case.
}\label{f:nu025}
\end{figure}

\section*{Acknowledgments}
We thank the Editor, Associate Editor and two anonymous referees for comments and suggestions. T.S. Prass gratefully acknowledges the support of FAPERGS (ARD 01/2017, Processo 17/2551-0000826-0).

\section*{References}
\bibliographystyle{myjmva}
\bibliography{ref_dcca}

\section*{Appendix: Proof of Theorems}

\begin{proof}[\bf {Proof of Proposition \ref{inv}}]
Let $h_{\ell+1}(t) := \sum_{j = 1}^t j^\ell$, $\ell \in \N$. Observe that $h_{\ell+1}(t)$ is the generalized harmonic number and hence, for each $\ell \in \N$,  $h_{\ell}(t)$ is a polynomial with order $\ell+1$ given by
\[
h_{\ell+1}(t) = \frac{ b_{\ell+1}(-1)^\ell}{\ell+1} + \frac{1}{\ell+1}B_{\ell+1}(t+1),
\]
where $b_\ell$ are the Bernoulli numbers and $B_\ell(\cdot)$ are the Bernoulli polynomials defined, respectively, through
\[
b_{\ell} = \sum _{\tau=0}^{\ell}\frac {1}{\tau+1}\sum _{v=0}^{\tau}{\tau \choose v}(-1)^{v}(v+1)^{\ell},  \qquad B_{\ell}(t) := \sum _{\tau=0}^{\ell}{\binom {\ell}{\tau}}b_{\tau}t^{\ell-\tau}.
\]
Observe further that $p_{k}(j) = \sum_{\ell = 0}^{\nu_k} a_{k,\ell} j^\ell$ so that we can write
\[
R_{k,t}=\sum_{j=1}^tX_{k,j}+\sum_{j=1}^tp_{k}(j) = \sum_{j=1}^tX_{k,j} + \sum_{\ell = 0}^{\nu_k} a_{k,\ell} h_{\ell+1}(t) = \sum_{j=1}^tX_{k,j} + p_{k}^\ast(t).
\]
where $p_{k}^\ast(t)$ is a polynomial of order $\nu_k+1$. Upon denoting $T_{k,j}=p_{k}(j)$ and $T_{k,j}^\ast=p_{k}^\ast(j)$, we have
$\bs{Y}_{k,n}^{\lr{1}}=\bs{X}_{k,n}^{\lr{1}}+\bs T_{k,n}^{\lr{1}}$,
\begin{equation*}
\bs{R}_{k,n}^{\lr{1}} = J_n\Big[\bs{X}_{k,n}^{\lr{1}}+\bs T_{k,n}^{\lr{1}}\Big] = J_n\bs{X}_{k,n}^{\lr{1}}  + \bs{T}_{k,n}^{\ast\lr{1}},
 \qquad \bs{R}_{k,m+i}^{\lr{i}}  =  J_{m+i}^{\lr{i}}\bs{X}_{k,m+i}^{\lr{1}}+ \bs{T}_{k,m+i}^{\ast\lr{i}},
\end{equation*}
$i \in\{ 1,\dots,n-m\}$. Now, with $D_{m+1}$, $P_{m+1}$ and $Q_{m+1}$ as in \eqref{eq:XPQ}, let $\bs{\Eps}_{k,i}^X$ and $\bs{\Eps}_{k,i}^Y$ denote the detrended walk based on $\{X_{k,t}\}_{t\in\Z}$ and $\{Y_{k,t}\}_{t\in\Z}$, respectively.
Since $P_{m+1}\bs{T}_{k,m+i}^{\ast\lr{i}}=\bs{T}_{k,m+i}^{\ast\lr{i}}$, it follows that
\[
\bs{\Eps}_{k,i}^Y=Q_{m+1}\bs{R}_{k,m+i}^{\lr{i}}=Q_{m+1}\Big[J_{m+i}^{\lr{i}}\bs{X}_{k,m+i}^{\lr{1}}+\bs T_{k,m+i}^{\ast\lr{i}}\Big] =Q_{m+1}J_{m+i}^{\lr{i}}\bs{X}_{k,m+i}^{\lr{1}}=\bs{\Eps}_{k,i}^X,
\]
and the result follows from \eqref{eq:dfak_i} and \eqref{eq:dcca_i}.
\end{proof}

\begin{proof}[\bf {Proof of Lemma \ref{lem:eps}}]
  Assume $k \in\{1,2\}$ and $i \in\{ 1,\ldots,n-m\}$.  Let $\mX_{k,i}$ be the vector defined as
\begin{equation}\label{eq:mX}
\mX_{k,i} := \bs{R}_{k,m+i}^{\lr{i}}  - R_{k,i}1_{m+1} = \bigg(0, X_{k,i+1}, \ldots, \sum_{t=i+1}^{m+i}X_{k,t}\bigg)^\top
\end{equation}
and $\mathcal{M}:=\{1, \ldots, n-m\}$. For any $1 \leq \tau \leq n-m$, $h\geq 0$ and $\bs{i} = (i_1, \ldots, i_\tau) \in \mathcal{M}^\tau$, let
\[
\bs{i}+h := (i_1+h, \ldots, i_\tau+h), \qquad \mQ_{k,\bs{i}} := \big(\mX_{k,i_1}^\top Q_{m+1},\ldots, \mX_{k,i_\tau}^\top Q_{m+1}\big),
\]
 with $Q_{m+1}$ be given by \eqref{eq:XPQ}. Hence, from the joint stationarity of $\{X_{1,t}\}_{t\in\Z}$ and $\{X_{2,t}\}_{t\in\Z}$,
\begin{equation}\label{eq:g12}
(\mQ_{1,\bs{i}}, \mQ_{2,\bs{j}})   \overset{d}{=} (\mQ_{1,\bs{i}+h}, \mQ_{2,\bs{j}+h}), \quad  \tau, \ell \in \mathcal{M}, \quad \bs{i} \in \mathcal{M}^\tau,  \quad \bs{j}\in \mathcal{M}^\ell,
\end{equation}
and all $h \geq 0$ for which the $(\mQ_{1,\bs{i}+h}, \mQ_{2,\bs{j}+h})$ is well defined. The result then follows by observing that, since $\E(X_{k,t})$ does not depend on $t$, from \eqref{eq:mX}, $
Q_{m+1}R_{k,i}1_{m+1} = 0_{m+1}$,  which implies $Q_{m+1}\mX_{k,i} = Q_{m+1}\bs{R}_{k,m+i}^{\lr{i}}  =  \bs{\Eps}_{k,i}$.
\end{proof}

\begin{proof}[\bf {Proof of Corollary \ref{lem:sta}}]
The result follows from \eqref{eq:dfak_i}, \eqref{eq:dcca_i} and Lemma \ref{lem:eps}.
\end{proof}

\begin{proof}[\bf {Proof of Theorem \ref{thm:expec}}]
  Assume $k,k_1,k_2 \in\{1,2\}$, $0 < m <n$, $0 < i < n-m$ and $0 \leq h, h_1, h_2 < n-m$.

\noindent The stationarity assumption for $\{X_{k,t}\}_{t\in\Z}$ and \eqref{eq:rt} imply that $Q_{m+1}\E\big(\bs{R}_{k,m+1}^{\lr{1}} \big) = 0_{m+1}$ so that
\begin{equation}\label{eq:traceCov}
Q_{m+1}\E\Big(\bs{R}_{k_1,m+1}^{\lr{1}}\big[\bs{R}_{k_2,m+1}^{\lr{1}}\big]^\top \Big) = Q_{m+1}\cov\Big(\bs{R}_{k_1,m+1}^{\lr{1}}, \bs{R}_{k_2,m+1}^{\lr{1}}\Big) = Q_{m+1}\Sigma_{k_1,k_2}.
\end{equation}
From Lemma \ref{lem:sta},  $\E(\bs{\Eps}_{k_1,i}^\top \bs{\Eps}_{k_2,i}) = \E(\bs{\Eps}_{k_1,1}^\top \bs{\Eps}_{k_2,1})$  so that, from \eqref{eq:epsk} and the properties of trace,
\begin{equation}
  \E(\bs{\Eps}_{k_1,i}^\top \bs{\Eps}_{k_2,i})
  = \E\Big( \trace \big(\big[\bs{R}_{k_1,m+1}^{\lr{1}}\big]^\top Q_{m+1}\bs{R}_{k_2,m+1}^{\lr{1}}\big)\Big) =\trace \Big(Q_{m+1}\E\big(\bs{R}_{k_1,m+1}^{\lr{1}}\big[\bs{R}_{k_2,m+1}^{\lr{1}}\big]^\top \big)\Big).
  \label{eq:Eee}
\end{equation}
From \eqref{eq:SG}, the definition of $K_{m+1}$ and the properties of the trace,  \eqref{eq:traceCov} and \eqref{eq:Eee} imply \eqref{eq:Edfai}, when  $k_1 = k_2 = k$, and  \eqref{eq:Edccai}, when $k_1 = 1$ and $k_2 = 2$.

In order to derive $\gamma_{k,DFA}$ and $\gamma_{DCCA}$, first let
\begin{align}
 \Lambda_{k_1,k_2}(h)& := \E\Big(\big[\bs{X}_{k_1,m+1}^{\lr{1}} \otimes \bs{X}_{k_1,m+1+h}^{\lr{1}}\big]\big[ \bs{X}_{k_2,m+1}^{\lr{1}}\otimes \bs{X}_{k_2,m+1+h}^{\lr{1}}\big]^\top \Big) \label{eq:Gammao},\\
 \Upsilon_{k_1,k_2}(h) & : = \E\Big(\big[\bs{R}_{k_1,m+1}^{\lr{1}}\otimes\bs{R}_{k_1,m+1+h}^{\lr{h+1}}\big]\big[ \bs{R}_{k_2,m+1}^{\lr{1}}\otimes\bs{R}_{k_2,m+1+h}^{\lr{h+1}}\big]^\top \Big), \label{eq:covR1R2o}
\end{align}
and observe that,  from \eqref{eq:RRi} and  the properties of the Kronecker product,
\[
\bs{R}_{k,m+1}^{\lr{1}}\otimes\bs{R}_{k,m+1+h}^{\lr{h+1}} =
  \Big[J_{m+1}\otimes J_{m+1+h}^{\lr{h+1}}\Big]\Big[\bs{X}_{k,m+1}^{\lr{1}}\otimes\bs{X}_{k,m+1+h}^{\lr{1}}\Big], \quad k\in\{1,2\}
\]
so that, from \eqref{eq:Gammao} and \eqref{eq:covR1R2o},
\begin{equation}\label{eq:GammaGo}
 \Upsilon_{k_1,k_2}(h)   =  \Big[J_{m+1}\otimes J_{m+1+h}^{\lr{h+1}}\Big] \Lambda_{k_1,k_2}(h)\Big[J_{m+1}\otimes J_{m+1+h}^{\lr{h+1}}\Big]^\top .
\end{equation}
 Notice that $\Lambda_{k_1,k_2}(h)$ is a $[(m+1)(m+1+h)] \times [(m+1)(m+1+h)]$  block matrix, for which the  $(r,s)$th element in the $(p,q)$th block is  $\E(X_{k_1,p} X_{k_1,r} X_{k_2,q} X_{k_2,s})$, $1\leq p,q \leq m+1$, $1\leq r,s\leq m+1+h$. Moreover, under the assumption  $\E(X_{k,t}) = 0$,
\begin{align*}
 \Big[\big[\Lambda_{k_1,k_2}(h)\big]^{p,q}\Big]_{r,s}   =  \ka_{k_1, k_2}(p,r,q,s) + \gamma_{k_1}(r-p)\gamma_{k_2}(s-q) +  \gamma_{k_1,k_2}(q-p)\gamma_{k_1,k_2}(s-r) +  \gamma_{k_1,k_2}(s-p)\gamma_{k_1,k_2}(q-r), \quad
\end{align*}
$1\leq p,q \leq m+1$, $1\leq r,s\leq m+1+h$, where  $\gamma_{k_1, k_2}$  is the cross-covariance function associated to $\{X_{k_1,t}\}_{t\in\Z}$ and $\{X_{k_2,t}\}_{t\in\Z}$,  $\gamma_k := \gamma_{k,k}$ and  $\ka_{k_1, k_2}(p,r,q,s)$ is the joint cumulant of $X_{k_1,p}, X_{k_1,r},X_{k_2,q}, X_{k_2,s}$. Furthermore, by letting  $\Gamma_{k_1,k_2}^{h_1,h_2}$, $\Gamma_k^{h_1, h_2} $ and $\K_{k_1,k_2}(h)$ be the matrices defined in \eqref{eq:SG} - \eqref{eq:Gamma12} and observing that $\big[\Gamma_{k}^{0,h}\big]^\top  = \Gamma_{k}^{h,0}$,  one can write
\[
\Lambda_{k_1,k_2}(h) = \K_{k_1,k_2}(h) + \vc\big(\Gamma_{k_1}^{h,0}\big)\vc\big(\Gamma_{k_2}^{h,0}\big)^\top  + \Gamma_{k_1,k_2}^{0,0} \otimes\Gamma_{k_1,k_2}^{h,h}  + \mathcal{G}_{k_1,k_2}(h),
\]
where $\mathcal{G}_{k_1,k_2}(h)$ is $[(m+1)(m+1+h)] \times [(m+1)(m+1+h)]$ block matrix given by
\[
\mathcal{G}_{k_1,k_2}(h) = \big[\mathcal{G}_1 \ \cdots  \ \mathcal{G}_{m+1}\big], \quad
 \mathcal{G}_{p} := \Gamma_{k_1,k_2}^{0,h} \otimes G_{p}, \quad
G_{p} := \begin{bmatrix}
\big[\Gamma_{k_1,k_2}^{h_1, h_2}]_{1,p}\\
\vdots\\
\big[\Gamma_{k_1,k_2}^{h_1, h_2}]_{m+1+h_1,p}
\end{bmatrix}, \quad 1 \leq p \leq m+1.
\]

Now,  from the properties of the Kronecker product and $Q_{m+1}$,
\begin{align*}
  (\bs{\Eps}_{k_1,1}^\top \bs{\Eps}_{k_2,1})(\bs{\Eps}_{k_1,1+h}^\top \bs{\Eps}_{k_2,1+h})
  & =   (\bs{\Eps}_{k_1,1}\otimes \bs{\Eps}_{k_1,1+h})^\top  (\bs{\Eps}_{k_2,1}\otimes\bs{\Eps}_{k_2,1+h}),\\
(\bs{\Eps}_{k,1}\otimes \bs{\Eps}_{k,1+h})  & =  \big(Q_{m+1}\otimes Q_{m+1}\big)\big(\bs{R}_{k,m+1}^{\lr{1}} \otimes \bs{R}_{k,m+1+h}^{\lr{1+h}}\big),\\
\big(Q_{m+1}\otimes Q_{m+1}\big)^\top \big(Q_{m+1}\otimes Q_{m+1}\big) & =  \big(Q_{m+1}\otimes Q_{m+1}\big),
\end{align*}
so that, by letting $\mR_{k,1,1+h} := \bs{R}_{k,m+1}^{\lr{1}} \otimes \bs{R}_{k,m+1+h}^{\lr{1+h}}$,
\begin{align}
  \E\big( [\bs{\Eps}_{k_1,1}^\top \bs{\Eps}_{k_2,1}][\bs{\Eps}_{k_1,1+h}^\top \bs{\Eps}_{k_2,1+h}]\big)
      &= \E\big( \trace(\mR_{k_1,1,1+h}^\top [Q_{m+1}\otimes Q_{m+1}]^\top [Q_{m+1}\otimes Q_{m+1}]\mR_{k_2,1,1+h})\big)\nonumber\\
  &= \trace\big([Q_{m+1}\otimes Q_{m+1}]\E(\mR_{k_1,1,1+h}\mR_{k_2,1,1+h}^\top )\big)\nonumber\\
  & =  \trace\big([Q_{m+1}\otimes Q_{m+1}]\Upsilon_{k_1,k_2} (h)\big).\label{eq:EEEE}
  \end{align}
Hence, the properties of trace, \eqref{eq:Eee}, \eqref{eq:GammaGo} and \eqref{eq:EEEE}, imply that
\[
\cov(\bs{\Eps}_{k_1,1}^\top \bs{\Eps}_{k_2,1},\bs{\Eps}_{k_1,h+1}^\top \bs{\Eps}_{k_2,h+1} ) =  \trace\big(K_{m+1}^\otimes(h)\Lambda_{k_1,k_2}(h)\big) - \big[\trace\big(K_{m+1}\Gamma_{k_1,k_2}\big)\big]^2
\]
with  $K_{m+1}$ and $K_{m+1}^\otimes(h)$ defined in \eqref{eq:Km} and \eqref{eq:Kmo}.

\noindent Finally, observe that
\begin{align*}
\trace\Big(\big[K_{m+1}\otimes K_{m+1}(h)\big]&\vc\big(\Gamma_{k_1}^{h,0}\big)\vc\big(\Gamma_{k_2}^{h,0}\big)^\top \Big) =
\trace\Big(\vc\big(\Gamma_{k_2}^{h,0}\big)^\top \big[K_{m+1}\otimes K_{m+1}(h)\big]\vc\big(\Gamma_{k_1}^{h,0}\big)\Big)\\
& = \trace\Big( \Gamma_{k_2}^{0,h}K_{m+1}(h)\Gamma_{k_1}^{h,0}K_{m+1}\Big) = \trace\Big( K_{m+1}\Gamma_{k_1}^{0,h}K_{m+1}(h)\Gamma_{k_2}^{h,0}\Big),
\end{align*}
\[
K_{m+1}(h)\Gamma_{k_1,k_2}^{h,h} = \begin{bmatrix} 0_{h,h} & 0_{h,m+1}\\
0_{m+1, h} & K_{m+1}\Gamma_{k_1,k_2}^{0,0}\\
 \end{bmatrix}\implies \trace\Big(K_{m+1}(h)\Gamma_{k_1,k_2}^{h,h}\Big) = \trace\Big(K_{m+1}\Gamma_{k_1,k_2}^{0,0}\Big)
\]
and
\begin{align*}
  \trace\big(K_{m+1}^\otimes(h) \mathcal{G}_{k_1,k_2}(h)\big)
  & = \sum_{p = 1}^{m+1} \sum_{s = 1}^{m+1+h} \Big[K_{m+1}\Gamma_{k_1,k_2}^{0,h}\Big]_{p,s}\big[ K_{m+1}(h)G_{p}\big]_s= \sum_{p = 1}^{m+1} \sum_{s = 1}^{m+1+h} \Big[K_{m+1}\Gamma_{k_1,k_2}^{0,h}\Big]_{p,s}\Big[ K_{m+1}(h)\Gamma_{k_1,k_2}^{h,0}\Big]_{s,p}\\
  & = \trace\Big(K_{m+1}\Gamma_{k_1,k_2}^{0,h} K_{m+1}(h)\Gamma_{k_1,k_2}^{h,0}\Big),
\end{align*}
so that
\begin{align}
\trace\big(K_{m+1}^\otimes(h) \Lambda_{k_1,k_2}(h)\big)& = \trace\big(K_{m+1}^\otimes(h)\K_{k_1,k_2}(h)\big) + \trace\Big( K_{m+1}\Gamma_{k_1}^{0,h}K_{m+1}(h)\Gamma_{k_2}^{h,0}\Big) + \Big[\trace\Big(K_{m+1} \Gamma_{k_1,k_2}^{0,0}\Big)\Big]^2
  + \trace\Big(K_{m+1}\Gamma_{k_1,k_2}^{0,h} K_{m+1}(h)\Gamma_{k_1,k_2}^{h,0}\Big).\label{eq:traceKL}
\end{align}
From \eqref{eq:Edfai},  \eqref{eq:traceKL} implies \eqref{eq:Vardfai} and \eqref{eq:Covdfai} when $k_1 = k_2 = k$,  and \eqref{eq:Vardccai} and \eqref{eq:Covdccai} when $k_1 = 1$ and $k_2 = 2$.
\end{proof}

Lemma \ref{lem:convP} {provides} sufficient conditions for the sample mean to converge in probability/almost surely to the process' mean and it is necessary to the proof of Theorem \ref{thm:convDCCA}.

\begin{lemma}\label{lem:convP}
  Let $\{Y_t\}_{t\in\Z}$ be a weakly stationary process and $\gamma(h) := \cov(Y_t, Y_{t+h})$, $h\in \Z$. Then
\begin{equation}\label{eq:xbar}
\gamma(h) \to 0, \as h \to \infty  \quad \implies \quad  \frac{1}{n}\sum_{i=1}^n Y_i \overset{P}{\To} \E(Y_t), \as n\to \infty,
\end{equation}
\begin{equation}\label{eq:xbaras}
\sum_{h = 1}^\infty \frac{|\gamma(h)|}{h^q} < \infty, \quad \mbox{for some }\ 0 \leq q < 1, \quad \implies \quad  \frac{1}{n}\sum_{i=1}^n Y_i\overset{a.s}{\To} \E(Y_t), \as n\to \infty.
\end{equation}
\end{lemma}

\begin{proof}[\bf {Proof of Lemma \ref{lem:convP}}]
Let $\bar{Y} := \frac{1}{n}\sum_{i=1}^n Y_i$. If $\gamma(h) \to 0$, for any $\eps >0$, there exist $h_0 > 0$ and $n_0 > h_0$ such that
$|\gamma(h)| < \frac{\eps}{6}$, for all $h > h_0$, $\frac{1}{n}\gamma(0) < \frac{\eps}{3}$, and $\frac{2}{n}\sum_{j = 1}^{h_0} |\gamma(j)| < \frac{\eps}{3}$, for all $n > n_0$. Thus,
\[
\big|\var(\bar{Y})| \leq \frac{1}{n}\gamma(0) + \frac{2}{n}\sum_{h = 1}^{h_0} |\gamma(h)| + \frac{2}{n}\sum_{h = h_0+1}^{n-1} |\gamma(h)|
 \leq \frac{\eps}{3}  + \frac{\eps}{3} + \frac{\eps}{3}\frac{(n-h_0-1)}{n} \leq \eps.\\
\]
Since $\eps$ is arbitrary, $\var(\bar{Y}) \to 0$, as $n\to \infty$, and \eqref{eq:xbar} follows immediately from Chebyshev's inequality. Now, suppose that
\[
\sum_{h = 1}^\infty \frac{|\gamma(h)|}{h^q} < \infty, \quad \mbox{for some }\ 0 \leq q < 1,
\]
and observe that the stationarity of $\{Y_t\}_{t\in\Z}$ implies that
\[
\sum_{n = 1}^\infty \frac{\var(Y_n)[\ln(n)]^2}{n^2} = \gamma(0)\sum_{n = 1}^\infty \frac{[\ln(n)]^2}{n^2} < \infty \aand \underset{n\geq 1}{\sup}\{\cov(Y_n, Y_{n+h})\} \leq |\gamma(h)|.
\]
Hence, by letting $\rho_h := |\gamma(h)|$, $h \geq 1$ and  $b_n := n$, the result follows from theorem 1 in \citet{Huea2008}.
\end{proof}

\noindent Observe that any stationary processes with absolutely summable covariance function satisfy  \eqref{eq:xbaras} with $q = 0$. Moreover, any stationary process for which $\gamma(h) \sim \ell h^{-\beta}$, $0 < \beta < 1$, where $\ell \in \R$, satisfy \eqref{eq:xbaras} for any $1- \beta < q <1$.

\begin{proof}[\bf {Proof of Theorem \ref{thm:convDCCA}}]
From Corollary \ref{lem:sta}, $\big\{f_{k,DFA}^2(m,i)\big\}_{i=1}^{n-m}$ and $\big\{f_{DCCA}(m,i)\big\}_{i=1}^{n-m}$, are strictly stationary processes.   Hence, the results for $F^2_{k,DFA}(m)$ and $F_{DCCA}(m)$ follow from  Lemma \ref{lem:convP} upon observing that,  from \eqref{eq:c1} and  \eqref{eq:c2},
  \[
F^2_{k,DFA}(m)= \frac{1}{n-m}\sum_{i = 1}^{n-m}f^2_{k,DFA}(m,i) = \frac{1}{N}\sum_{i=1}^N Y_{k,i}, \quad Y_{k,i} = f^2_{k,DFA}(m,i) \aand  N = n-m,
\]
\[
F_{DCCA}(m) = \frac{1}{n-m}\sum_{i = 1}^{n-m} f_{DCCA}(m,i) = \frac{1}{N}\sum_{i = 1}^{N} Z_{i}, \quad \quad Z_{i} = f^2_{DCCA}(m,i) \aand  N = n-m.
\]
The results for $\rho_{DCCA}(m)$ follow upon observing that, from \eqref{eq:rho_dcca},
\[
\rho_{DCCA}(m) = \frac{F_{DCCA}(m)}{F_{1,DFA}(m)F_{2,DFA}(m)},  \quad 0< m< n,
\]
and  applying the continuous mapping theorem.
\end{proof}

\begin{proof}[\bf {Proof of Proposition \ref{lem:cond}}]
Notice that,  for any  $k_1,k_2\in\{ 1,2\}$ and $h > m+1$,
{\[
K_{m+1}\Gamma_{k_1,k_2}^{0,h} = \begin{bmatrix} K_{m+1}\Gamma_{k_1,k_2}^{0,h-(m+1)} & K_{m+1}H_{k_1,k_2}(h)\\
 \end{bmatrix}, \quad
K_{m+1}(h)\Gamma_{k_1,k_2}^{h,0} = \begin{bmatrix}  0_{h,m+1}\\
 K_{m+1}H_{k_1,k_2}(-h)\\
 \end{bmatrix},
\]}
where
\[
H_{k_1, k_2}(h) := \begin{bmatrix}
\gamma_{k_1, k_2}(h) & \cdots  & \gamma_{k_1, k_2}(m+h) \\
\vdots & \ddots & \vdots \\
\gamma_{k_1, k_2}(h-m) & \cdots & \gamma_{k_1, k_2}(h) \\
\end{bmatrix}, \quad h \in \Z.
\]
Hence, $\gamma_{k_1,k_2}(h) \to 0$, as  $|h| \to \infty$,  implies
\[
\trace\Big(K_{m+1}\Gamma_{k_1,k_2}^{0,h}K_{m+1}(h)\Gamma_{k_1,k_2}^{h,0}\Big) = {\trace\big(K_{m+1}H_{k_1,k_2}(h)K_{m+1}H_{k_1,k_2}(-h)\big) } \to 0.
\]
Now, observe that $K_{m+1}^\otimes(h)\K_{k_1,k_2}(h)$ is a block matrix for which the $(p,q)$th block is given by
\begin{align*}
\big[K_{m+1}^\otimes(h)\K_{k_1,k_2}(h)\big]^{p,q} & = \sum_{\ell = 1}^{m+1} [K_{m+1}]_{p,\ell}K_{m+1}(h)[\K_{k_1,k_2}(h)]^{\ell,q}\\
&  = [K_{m+1}]_{p,\ell}
\begin{bmatrix}
0_{h,h} & 0_{h,m+1}\\
K_{m+1}\Upsilon_{k_1,k_2}^{\ell,q}(h) &  K_{m+1}\mathcal{C}_{k_1,k_2}^{\ell,q}(h) \\
\end{bmatrix}, \quad 1 \leq p,q \leq m+1,
\end{align*}
with
\[
\Big[\Upsilon_{k_1,k_2}^{\ell,q}(h)\Big]_{\tau,s} = \ka_{k_1,k_2}(\ell,h+\tau,q,s), \quad 1 \leq \tau \leq m+1, \quad 1 \leq s \leq h,
\]
\[
\Big[\mathcal{C}_{k_1,k_2}^{\ell,q}(h)\Big]_{\tau,s} = \ka_{k_1,k_2}(\ell,h+\tau,q,h+s), \quad 1 \leq \tau,s \leq m+1.
\]
It follows that
\begin{align*}
  \trace\big(K_{m+1}^\otimes(h)\K_{k_1,k_2}(h)\big)
  & = \sum_{p = 1}^{m+1}\trace\big(\big[K_{m+1}^\otimes(h)\K_{k_1,k_2}(h)\big]^{p,p}\big)
    = \sum_{p = 1}^{m+1}[K_{m+1}]_{p,p}\trace\big(K_{m+1}\mathcal{C}_{k_1,k_2}^{p,p}(h)\big)\\
  &  = \sum_{p = 1}^{m+1}\sum_{q = 1}^{m+1}[K_{m+1}]_{p,p}\big[K_{m+1}\mathcal{C}_{k_1,k_2}^ n{p,p}(h)\big]_{q,q} = \sum_{p = 1}^{m+1} \sum_{q = 1}^{m+1}  \sum_{\tau = 1}^{m+1}[K_{m+1}]_{p,p}\big[K_{m+1}\big]_{q,\tau}\ka_{k_1,k_2}(p,h+\tau,p,h+q).
\end{align*}
Hence $\ka_{k_1,k_2}(p,h+\tau,p,h+q) \to 0$, as $h\to \infty$, implies
\[
\trace\big(K_{m+1}^\otimes(h)\K_{k_1,k_2}(h)\big) \To 0, \as h\to \infty.
\]
Therefore, the result follows.
\end{proof}

\begin{proof}[\bf {Proof of Lemma \ref{lem:alpham}}]
Since for any $\nu\geq0$, $Q_{m+1}1_{m+1} = 0$,  it follows that $K_{m+1}$ is a symmetric matrix satisfying $[K_{m+1}]_{r,s} = 0$, if $r=1$ or $s=1$.
Also,  $M_0 = M_0^\ast =  I_{m+1}$, so that $\trace (K_{m+1}M_0) = \trace (K_{m+1}M_0^\ast)  =  \trace(K_{m+1})$, while  $\trace (K_{m+1}M_m) = \trace (K_{m+1}M_m^\ast)=0$. For $1\leq j<m$, we have
\begin{equation*}
\trace (K_{m+1}M_j)  = \sum_{r = 1}^{m+1}\sum_{s = 1}^{m+1}[K_{m+1}]_{r,s}[M_j]_{s,r} =  2\!\!\!\sum_{r = 2}^{m+1-j}\!\![K_{m+1}]_{r,r+j},
\end{equation*}
and, analogously,
\begin{equation*}
  \alpha_j^{(m)} :=\trace (K_{m+1}M_j^\ast)   = \sum_{r = 2}^{m+1-j}\!\![K_{m+1}]_{r,r+j} = \frac{1}{2}\trace (K_{m+1}M_j).
\end{equation*}
Consider now the case $\nu=0$. In this case, for all $r,s > 1$,   the $(r,s)$th element in the matrix $K_{m+1}$ is given by
\begin{align*}
  [K_{m+1}]_{r,s}  &  = \sum_{i = r}^{m+1}\sum_{j = s}^{m+1}[Q_{m+1}]_{i,j}
                     = m+2-\max\{r,s\} -\sum_{i = r}^{m+1}\sum_{j = s}^{m+1}[P_{m+1}]_{i,j}\nonumber\\
  & = m+2-\max\{r,s\} - \frac{(m + 2- r) (m + 2 - s) \big[m^2 + 2 m + 3 (r - 1) (s - 1)\big]}{m (m + 1) (m + 2)}.
 \end{align*}
It follows that
\[
 \alpha_0^{(m)} = \trace \big(K_{m+1}\big)  = \sum_{r = 2}^{m+1}\Bigg\lbrace  m+2-r - \frac{(m + 2 - r)^2 \big[m^2 + 2 m + 3 (r - 1)^2\big]}{m (m + 1) (m + 2)} \bigg\rbrace  = \frac{m^2 + 2m -3}{15}.
\]
and
  \begin{align*}
    \alpha_j^{(m)} & =  \sum_{r = 2}^{m+1-j}\Bigg\lbrace m+2-(r+j) - \frac{(m + 2- r) (m + 2- r-j) \big[m^2 + 2 m + 3 (r - 1) (r+j - 1)\big]}{m (m + 1) (m + 2)}\Bigg\rbrace\\
    & = \frac{(m -j)(m + 1 -j)(m + 2 -j)}{m(m + 1)(m + 2)}\bigg[\frac{m^2 + 2m -3}{15}- \frac{j^2 + 3 j(m + 1)}{10}  \bigg] \sim \alpha_0^{(m)}, \as m\to \infty.
\end{align*}
This completes the proof.
\end{proof}

\begin{proof}[\bf {Proof of Theorem \ref{thm:asymp}}]
  Suppose that $\sum_{h\in \Z}|\gamma_{k}(h)| < \infty$.   Let $a_k,b_k >0$ be the two real constants satisfying $|\gamma_k(h)| \leq a_k e^{-b_k|h|}$, for all $h \in \Z$.   Then, for any $0 < h_0 < m$,  from \eqref{eq:gammask},
\[
  \frac{1}{m}\trace(K_{m+1}\Gamma_k) = \frac{\alpha_0^{(m)}}{m}\bigg[\gamma_k(0) + 2\sum_{h = 1}^{h_0-1}\beta_h^{(m)}\gamma_k(h) + 2\sum_{h = h_0}^{m-1}\beta_h^{(m)}\gamma_k(h)\bigg].
\]
Now, since $|\beta_h^{(m)}| \leq 1$ and $|\gamma_k(h)| \leq a_k e^{-b_k h}$,
\[
\bigg| \sum_{h = h_0}^{m-1}\beta_h^{(m)}\gamma_k(h)\bigg| \leq a_k\sum_{h = h_0}^{\infty} e^{-b_kh}   = a_k \frac{e^{-b_kh_0}}{1 - e^{-b_k}} < \frac{a_k}{b_k(1-e^{-b_k})}\frac{1}{h_0} = O(h_0^{-1}),
\]
uniformly in $m$.  From Lemma \ref{lem:alpham}, $\beta_h^{(m)} \to 1$, for all $0 \leq h < h_0$, so that
\[
\sum_{h = 1}^{h_0-1}\beta_h^{(m)}\gamma_k(h) =  \big[1-o(1)\big]\sum_{h = 1}^{h_0-1}\gamma_k(h) \sim \sum_{h = 1}^{h_0-1}\gamma_k(h) , \as m \to \infty.
\]
Now observe that $\E\big(f^2_{k,DFA}(m,1)\big) = m^{-1}\trace(K_{m+1}\Gamma_k)$ and, for arbitrary $h_0$,
\[
\frac{m^{-1}\trace(K_{m+1}\Gamma_k)}{m/15\sum_{h\in \Z}\gamma_k(h)} =  \frac{m^2 + 2m-3}{m^2}\bigg[\frac{\gamma_k(0) + 2[1-o(1)]\sum_{h = 1}^{h_0-1}\gamma_k(h)}{\sum_{h\in \Z}\gamma_k(h)} + O(h_0^{-1}) \bigg]
\]
so that the result follows by letting $h_0 \to \infty$ at a slower rate than $m$. A similar argument applies to
$\E(f_{DCCA}(m,1)) = \frac1m\trace(K_{m+1}\Gamma_{1,2})$.
\end{proof}

\begin{proof}[\bf {Proof of Corollary \ref{coro:linear}}]
From  \eqref{eq:gammasX},
\[
 \sum_{h\in \Z}\gamma_k(h) =  \tau_k^2\sum_{h\in \Z}\sum_{j\in\Z}\psi_{k,j}\psi_{k, j+h} =   \tau_k^2\bigg[\sum_{h\in \Z}\psi_{k,j}\bigg]^2, \quad k \in\{1,2\}
  \]
  and
  \[
  \sum_{h\in \Z}\gamma_{1,2}(h)  =  \tau_{1,2}\sum_{h\in \Z}\sum_{j\in\Z} \psi_{1,j}\psi_{2,j+h} =    \tau_{1,2} \sign\bigg(\sum_{h\in \Z}\sum_{j\in\Z} \psi_{1,j}\psi_{2,j+h}\bigg) \sqrt{\Big[\sum_{j\in \Z} \psi_{1,j}\Big]^2\Big[\sum_{j\in\Z}\psi_{2,j} \Big]^2}.
  \]
The result now follows immediately from Theorem \ref{thm:asymp}.
\end{proof}

\newpage


\section*{Supplementary material}

# On the behavior of the DFA and DCCA in trend-stationary processes
## Supplementary Material

Taiane Schaedler Prass[a,*], Guilherme Pumi[a]

[a]*Instituto de Matemática e Estatística - Universidade Federal do Rio Grande do Sul - Av. Bento Gonçalves, 9500, Porto Alegre - RS, Brazil.*

## 1. Examples and Monte Carlo Simulation Studies

In this section we describe the behavior of $\rho_{DCCA}(m)$ under different scenarios. Theorem 3.2 is applied when $n \to \infty$ and (31) and/or Corollary 4.1 are used for $m, n \to \infty$, in the non-trivial cases. For each example discussed, we present the results from a Monte Carlo simulation study on the finite sample performance of the $F^2_{DFA}(m)$, $F_{DCCA}(m)$ and $\rho_{DCCA}(m)$. The codes used in the simulation were implemented by the authors and are available in the package "DCCA" for R [3]. All time series were created with sample size $n = 2{,}000$ and $1{,}000$ replications were performed. For each replica, $\rho_{DCCA}(m)$ was calculated using (9), for $3 \le m \le 101$ with increments of size 2. For better visualization, the boxplots of the estimated values, for different values of $m$, and the theoretical expected values $\mathrm{E}(F^2_{DFA}(m))$ and $\mathrm{E}(F_{DCCA}(m))$, and the limit $\rho_{\mathcal{E}}(m)$ for which the coefficient $\rho_{DCCA}(m)$ converges to, as $n \to \infty$, are presented in a single graphic. We discuss the simulation results at the end of this section.

### 1.1. Uncorrelated Processes

Suppose that $\{X_{k,t}\}_{t \in \mathbb{Z}}$, $k \in \{1, 2\}$, are two stationary uncorrelated processes for which

$$\gamma_k(h) \longrightarrow 0 \quad \text{and} \quad \kappa_{k_1,k_2}(p, \tau + h, p, q + h) \longrightarrow 0, \quad \text{as} \quad h \to \infty, \quad k, k_1, k_2 \in \{1, 2\},$$

for any $p, q, \tau > 0$ fixed. It follows that

$$\Gamma_k = \mathrm{Cov}(X^{\langle 1 \rangle}_{k,m+1}, X^{\langle 1 \rangle}_{k,m+1}), \quad k \in \{1, 2\}, \quad \text{and} \quad \Gamma_{1,2} = 0_{m+1, m+1}, \quad \forall\, m > 0,$$

and hence

$$F_{DCCA}(m) \xrightarrow{P} 0 \quad \text{and} \quad \rho_{DCCA}(m) \xrightarrow{P} 0 \quad \text{as} \quad n \to \infty, \quad \forall\, m > 0.$$

Notice that the limits of $F_{DCCA}(m)$ and $\rho_{DCCA}(m)$ do not depend on the marginal behavior of $\{X_{k,t}\}_{t \in \mathbb{Z}}$, $k \in \{1, 2\}$.

Fig. 1 presents the simulation results considering uncorrelated time series. In Fig. 1(a) both time series are i.i.d. standard Gaussian sequences while in Fig. 1(b), $\{X_{k,t}\}_{t \in \mathbb{Z}}$, $k \in \{1, 2\}$ are given by

$$X_{1,t} = \phi X_{1,t-1} + \varepsilon_{1,t} \quad \text{and} \quad X_{2,t} = \varepsilon_{2,t} + \theta \varepsilon_{2,t-1}, \quad \varepsilon_{k,t} \underset{iid}{\sim} \mathcal{N}(0,1), \quad \varepsilon_{1,r} \perp \varepsilon_{2,s}, \forall\, r, s,$$

with $\phi = \theta = 0.6$, that is, $\{X_{1,t}\}_{t \in \mathbb{Z}}$ is an AR(1) process and $\{X_{2,t}\}_{t \in \mathbb{Z}}$ is an MA(1) process generated from two i.i.d. standard Gaussian sequences, independent of each other. To obtain the limits of $F^2_{k,DFA}(m)$, as $n \to \infty$, we observe that the $(r, s)$th element of $K_{m+1}$ is given by

$$[K_{m+1}]_{r,s} = m + 2 - \max\{r, s\} - \frac{(m + 2 - r)(m + 2 - s)[m^2 + 2m + 3(r-1)(s-1)]}{m(m+1)(m+2)},$$

---

*Corresponding author. Email address: taiane.prass@ufrgs.br

for any $1 \leq r, s \leq m+1$, and that

$$\gamma_1(h) = \frac{\phi^{|h|}}{1-\phi^2} \quad \text{and} \quad \gamma_2(h) = (1+\theta^2)I(h=0) + \theta I(|h|=1), \quad h \in \mathbb{Z},$$

where $I(\cdot)$ is the indicator function. Hence, a cumbersome but straightforward calculation yields

$$F^2_{1,DFA}(m) \xrightarrow[n\to\infty]{P} \frac{m^3}{15(1-\phi)^2(m^2+3m+2)} + \frac{O(m^4)}{m^4+3m^3+2m^2} \sim \frac{m}{15(1-\phi)^2}$$

and

$$F^2_{2,DFA}(m) \xrightarrow[n\to\infty]{P} \frac{(1+\theta)^2 m^2 + (2\theta^2 - 11\theta + 2)m - 3(\theta^2 - 3\theta + 1)}{15m} \sim \frac{(1+\theta)^2 m}{15}.$$

In order to explore the effect of the sample size on the $\rho_{DCCA}$, in Fig. 2 we present the results under the same data generating process as before, but considering fixed $m = 100$ and varying $n \in \{500, 600, \cdots, 10000\}$.

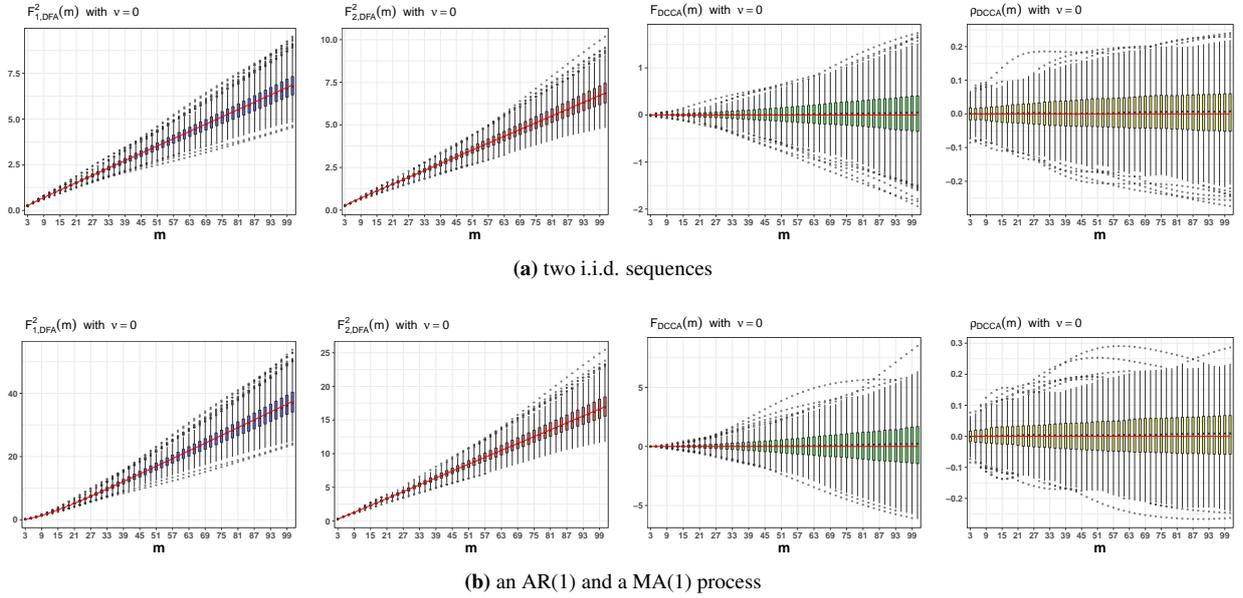

(a) two i.i.d. sequences

(b) an AR(1) and a MA(1) process

**Fig. 1:** Boxplots considering 1,000 replications and $m \in \{3, 5, \ldots, 101\}$, for two different scenarios under which there is no cross-correlation between the time series. From left to right, $F^2_{1,DFA}(m)$, $F^2_{2,DFA}(m)$, $F_{DCCA}(m)$ and $\rho_{DCCA}(m)$. In all cases, the red line represents the theoretical limit obtained by letting $n \to \infty$.

## 1.2. Bivariate white noise process
### 1.2.1. General model

Suppose that $\{(X_{1,t}, X_{2,t})\}_{t \in \mathbb{Z}}$ is a bivariate white noise, with $\mathrm{E}(X_{k,t}) = \mu_k$, $\mathrm{Var}(X_{k,t}) = \sigma_k^2$, $k \in \{1, 2\}$, $\mathrm{Cov}(X_{1,t}, X_{2,t}) = \sigma_{12}$ and

$$K_{k_1,k_2}(p, \tau+h, p, q+h) \longrightarrow 0, \quad \text{as} \quad h \to \infty, \quad k_1, k_2 \in \{1, 2\},$$

for any $p, q, \tau > 0$ fixed. It follows that

$$\Gamma_k = \sigma_k^2 I_{m+1}, \quad k \in \{1,2\}, \quad \text{and} \quad \Gamma_{1,2} = \sigma_{12} I_{m+1}, \quad \forall\, m > 0,$$

hence,

$$F^2_{k,DFA}(m) \xrightarrow[n\to\infty]{P} \left[\frac{1}{15}m + \frac{2}{15} - \frac{1}{5m}\right]\sigma_k^2 \sim \frac{\sigma_k^2}{15}m, \quad F_{DCCA}(m) \xrightarrow[n\to\infty]{P} \left[\frac{1}{15}m + \frac{2}{15} - \frac{1}{5m}\right]\sigma_{12} \sim \frac{\sigma_{12}}{15}m,$$

2

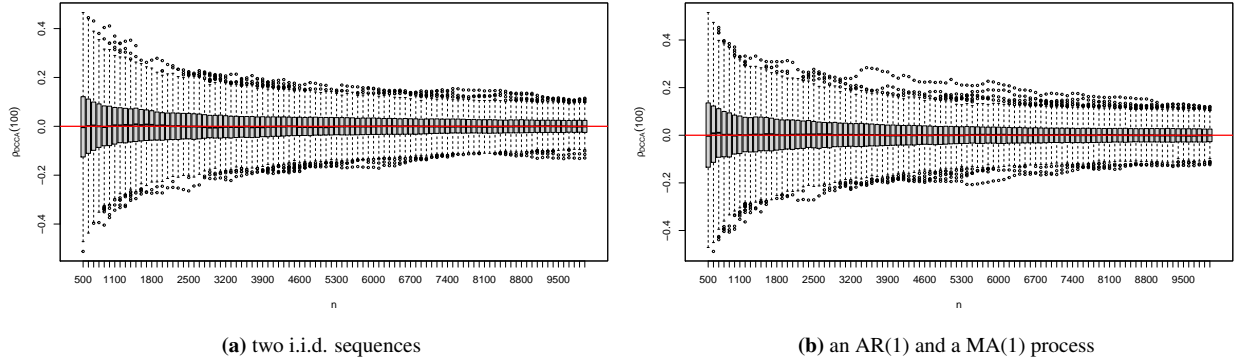

**(a)** two i.i.d. sequences

**(b)** an AR(1) and a MA(1) process

**Fig. 2:** Boxplots of the $\rho_{DCCA}(m)$ values, considering 1,000 replications and the same scenarios as in Fig. 1, for fixed $m = 100$ and sample sizes $n \in \{500, 600, \cdots, 10000\}$. In all cases, the red line represents the theoretical limit obtained by letting $n \to \infty$.

and
$$\rho_{DCCA}(m) \xrightarrow{P} \rho \quad \text{as} \quad n \to \infty, \quad \forall \, m > 0,$$
where $\rho = \text{Corr}(X_{1,t}, X_{2,t}) = \sigma_{12}/\sqrt{\sigma_1^2 \sigma_2^2}$.

Fig. 3 presents the simulation results considering time series which are cross-correlated only at lag $h = 0$. In Fig. 3(a) and (b) $\{(X_{1,t}, X_{2,t})\}_{t \in \mathbb{Z}}$ is an i.i.d. sample from a bivariate Gaussian distribution with $\text{Var}(X_{k,t}) = 1$, $k \in \{1, 2\}$ and $\text{Cov}(X_{1,t}, X_{2,t}) = \rho$, with $\rho = 0.5$ and $0.8$, respectively.

*1.2.2. Signal plus noise model*

A particular case of a bivariate white noise process is the signal plus noise model. Suppose that $\{X_{2,t}\}_{t \in \mathbb{Z}}$ is given by
$$X_{2,t} = \beta_0 + \beta_1 X_{1,t} + \varepsilon_t, \quad t \in \mathbb{Z}, \tag{1}$$
where $\{X_{1,t}\}_{t \in \mathbb{Z}}$ and $\{\varepsilon_t\}_{t \in \mathbb{Z}}$ are i.i.d. sequences with finite fourth moments, $\varepsilon_t \perp X_{1,s}$, $\forall t, s \in \mathbb{Z}$,
$$E(X_{1,t}) = \mu_1, \quad \text{Var}(X_{1,t}) = \sigma_1^2, \quad E(\varepsilon_t) = 0 \quad \text{and} \quad \text{Var}(\varepsilon_t^2) = \sigma_\varepsilon^2.$$

It follows that
$$\Gamma_k = a_k I_{m+1}, \quad k \in \{1, 2\}, \quad \text{and} \quad \Gamma_{1,2} = \beta_1 \sigma_1^2 I_{m+1}, \quad \forall \, m > 0,$$
where $a_1 = \sigma_1^2$ and $a_2 = \beta_1^2 \sigma_1^2 + \sigma_\varepsilon^2$. Moreover, it is easy to verify that, for $k_1, k_2 \in \{1, 2\}$,
$$K_{k_1,k_2}(p, \tau + h, p, q + h) = 0, \quad \text{if } p \neq \tau + h \text{ and } p \neq q + h.$$

Hence
$$F^2_{k,DFA}(m) \xrightarrow[n \to \infty]{P} \left[\frac{1}{15}m + \frac{2}{15} - \frac{1}{5m}\right] a_k \sim \frac{a_k}{15}m, \quad F_{DCCA}(m) \xrightarrow[n \to \infty]{P} \left[\frac{1}{15}m + \frac{2}{15} - \frac{1}{5m}\right] \beta_1 \sigma_1^2 \sim \frac{\beta_1 \sigma_1^2}{15}m,$$

and
$$\rho_{DCCA}(m) \xrightarrow{P} \frac{1}{\sqrt{1 + \sigma_\varepsilon^2/(\beta_1 \sigma_1)^2}} = \text{Corr}(X_{1,t}, X_{2,t}), \quad \text{as} \quad n \to \infty, \quad \forall \, m > 0.$$

Notice that, the smaller the ratio $\sigma_\varepsilon^2/(\beta_1 \sigma_1)^2$, the closer $\rho_{DCCA}(m)$ is to 1.

Fig. 3(c) and (d) present the simulation results for the scenarios where $\{X_{1,t}\}_{t \in \mathbb{Z}}$ is an i.i.d. standard Gaussian sequence and $\{X_{2,t}\}_{t \in \mathbb{Z}}$ is a signal plus noise process, defined through (1), with $\beta_0 = 3, \beta_1 = 2$ and $\{\varepsilon_t\}_{t \in \mathbb{Z}}$ a Gaussian i.i.d sequence with zero mean and variance $\sigma_\varepsilon^2 = 4$ and $64$, respectively. In these two scenarios, the corresponding limits are, respectively, $1/\sqrt{2}$ and $1/\sqrt{17}$.

3

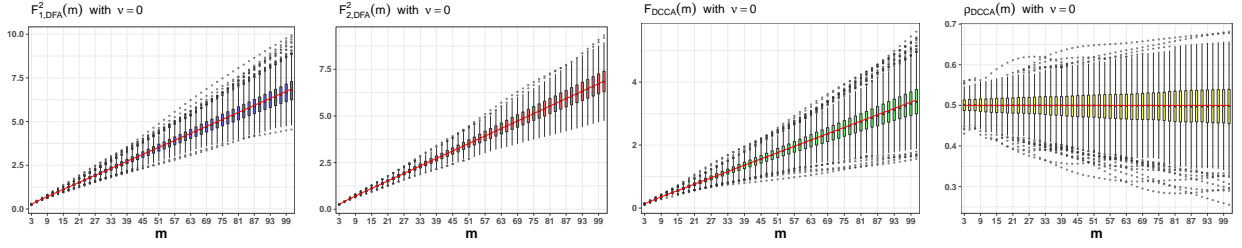
(a) Bivariate white noise process with $\rho = 0.5$

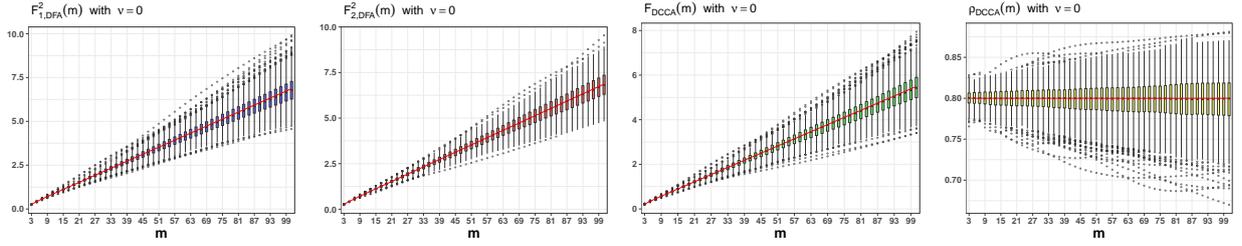
(b) Bivariate white noise process with $\rho = 0.8$

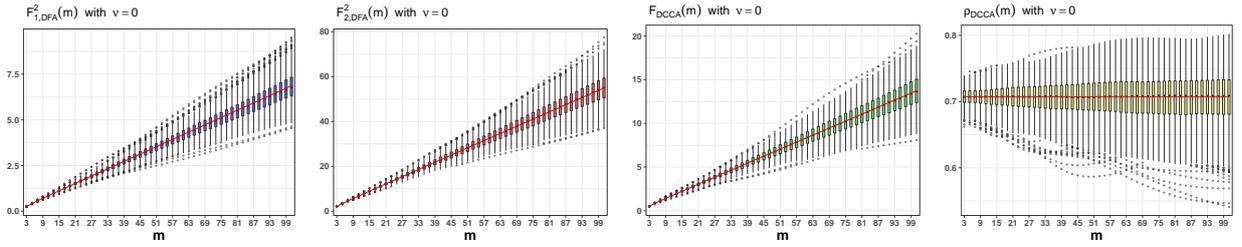
(c) Signal Plus Noise Model with $\sigma_\varepsilon^2 = 4$

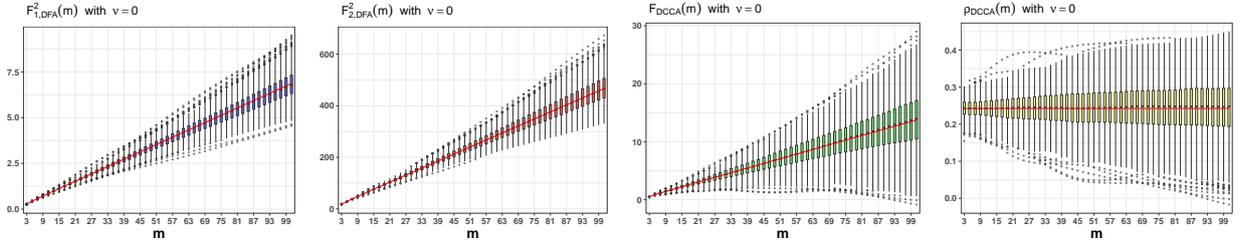
(d) Signal Plus Noise Model with $\sigma_\varepsilon^2 = 64$

**Fig. 3:** Boxplots considering 1,000 replications and $m \in \{3, 5, \ldots, 101\}$, for three different scenarios under which there is no cross-correlation at lag $h \neq 0$. From left to right, $F_{1,DFA}^2(m)$, $F_{2,DFA}^2(m)$, $F_{DCCA}(m)$ and $\rho_{DCCA}(m)$. In all cases, the red line represents the theoretical limit obtained by letting $n \to \infty$.

### 1.3. Short memory cross-correlated processes

#### 1.3.1. Dependence driven by a moving average structure

Suppose that $\{X_{1,t}\}_{t \in \mathbb{Z}}$ is an i.i.d. sequence with $\mathrm{E}(X_{1,t}) = \mu_1$ and $\mathrm{Var}(X_{1,t}) = \sigma_1^2$, and $\{X_{2,t}\}_{t \in \mathbb{Z}}$ is given by

$$X_{2,t} = \sum_{j=0}^{q} \theta_j X_{1,t-j}, \quad t \in \mathbb{Z}.$$

4

It follows that $\gamma_1(h) = \sigma_1 I(h=0)$,

$$\gamma_2(h) = \begin{cases} \sigma_1^2 \sum_{j=0}^{q-|h|} \theta_j \theta_{j+|h|}, & -q \leq h \leq q, \\ 0, & \text{otherwise}, \end{cases} \qquad \gamma_{1,2}(h) = \begin{cases} \sigma_1^2 \theta_h, & \text{if } 0 \leq h \leq q, \\ 0, & \text{otherwise}, \end{cases}$$

and since $\{X_{2,t}\}_{t\in\mathbb{Z}}$ satisfies (32), $\kappa_{k_1,k_2}(p, \tau+h, p, q+h) \longrightarrow 0$, as $h \to \infty$, for $k_1, k_2 \in \{1, 2\}$. Hence

$$\rho_{DCCA}(m) \xrightarrow{P} \frac{\theta_0 + \sum_{h=1}^{\tau} \beta_h^{(m)} \theta_h}{\sqrt{\sum_{h=0}^{q} \theta_h^2 + 2 \sum_{h=1}^{q} \sum_{i=0}^{q-h} \beta_h^{(m)} \theta_i \theta_{i+h}}}, \qquad \text{as} \quad n \to \infty. \tag{2}$$

with $\tau = \min\{q, m-1\}$, $\beta_h^{(m)} := \alpha_h^{(m)}/\alpha_0^{(m)}$. Furthermore, from Lemma 4.1, $\beta_h^{(m)} \to 1$, as $m \to \infty$, and $\sum_{h=0}^{q} \theta_h^2 + 2 \sum_{h=1}^{q} \sum_{i=0}^{q-h} \theta_i \theta_{i+h} = (\sum_{h=0}^{q} \theta_h)^2$, so that

$$\rho_{DCCA}(m) \xrightarrow{P} \text{sign}\left(\sum_{j=0}^{q} \theta_j\right), \qquad \text{as} \quad n, m \to \infty.$$

The same conclusion is achieved by observing that $\{X_{1,t}\}_{t\in\mathbb{Z}}$ and $\{X_{2,t}\}_{t\in\mathbb{Z}}$ satisfy (32) with $\varepsilon_{1,t} = \varepsilon_{2,t} = X_{1,t}$, for all $t \in \mathbb{Z}$, $\psi_{1,0} = 1$, $\psi_{1,j} = 0$, for all $j \neq 0$, $\psi_{2,j} = \theta_j$, $j \in \{0, \ldots, q\}$ e $\psi_{2,j} = 0$, otherwise, $\tau_1 = \tau_2 = \sigma_1^2$ and $\tau_{1,2} = 1$ and applying Corollary 4.1.

Fig. 4 presents the simulation results considering time series with non-zero cross-correlation for $h \geq 0$. In Fig. 4(a) the dependence is driven by a moving average structure as follows

$$X_{2,t} = X_{1,t} + \sum_{j=1}^{20} \theta_j X_{1,t-j}, \qquad \theta_j := \frac{21-j}{10}, \qquad X_{1,t} \underset{iid}{\sim} \mathcal{N}(0, 1). \tag{3}$$

In this case, a cumbersome but straightforward calculation yields

$$F_{DCCA}(m) \xrightarrow[n\to\infty]{P} \frac{22}{15}m + \frac{1038312}{5m^2} + \frac{2143922}{5(m+1)} - \frac{519156}{5(m+2)} - \frac{1617088}{5m} - \frac{1111}{15} \sim \frac{22}{15}m,$$

$$F_{1,DFA}^2(m) \xrightarrow[n\to\infty]{P} \frac{1}{15}m + \frac{2}{15} - \frac{1}{5m} \sim \frac{1}{15}m$$

and

$$F_{2,DFA}^2(m) \xrightarrow[n\to\infty]{P} \frac{484}{15}m + \frac{50588868}{25m^2} + \frac{105398282}{25(m+1)} - \frac{25294434}{25(m+2)} - \frac{159094419}{50m} - \frac{193369}{150} \sim \frac{22^2}{15}m.$$

Hence, upon applying a Taylor expansion of $\sqrt{a+x}$ around $x=0$ we obtain, for any $0 < \delta \leq 1$,

$$\rho_{DCCA}(m) \xrightarrow[n\to\infty]{P} \frac{1 + o(m^{-\delta})}{\sqrt{1 + o(m^{-\delta})}} \sim \frac{1 + o(m^{-\delta})}{1 + o(m^{-\delta})} \sim 1, \qquad \text{as} \quad m \to \infty.$$

*1.3.2. Dependence driven by an autoregressive structure*

Suppose that $\{X_{1,t}\}_{t\in\mathbb{Z}}$ is a white noise sequence, with $\text{E}(X_{1,t}) = \mu_1$ and $\text{Var}(X_{1,t}) = \sigma_1^2$, and $\{X_{2,t}\}_{t\in\mathbb{Z}}$ is given by

$$X_{2,t} = \sum_{j=1}^{p} \phi_j X_{2,t-j} + X_{1,t}, \qquad t \in \mathbb{Z}.$$

Observe that, if $\phi(z) := 1 - \phi_1 z - \cdots - \phi_p z^p$ has no roots in the unit circle, then $\{X_{2,t}\}_{t\in\mathbb{Z}}$ is stationary and can be written as $X_{2,t} = \sum_{j\in\mathbb{Z}} \psi_j X_{1,t-j}$, $t \in \mathbb{Z}$, with $\sum_{j\in\mathbb{Z}} |\psi_j| < \infty$, where $\{\psi_j\}_{j\in\mathbb{Z}}$ are the coefficients of $\psi(z) = (1 - \phi_1 z - \cdots - \phi_p z^p)^{-1} =$

$\phi^{-1}(z)$. In particular, if $\phi(z) \neq 0$, for all $|z| \leq 1$, then $\psi_j = 0$, for all $j < 0$, and $\{X_{2,t}\}_{t \in \mathbb{Z}}$ is causal. If $\phi(z) \neq 0$, for all $|z| \geq 1$, then $\psi_j = 0$, for all $j \geq 0$. Hence $\{X_{1,t}\}_{t \in \mathbb{Z}}$ and $\{X_{2,t}\}_{t \in \mathbb{Z}}$ satisfy (32) with $\varepsilon_{1,t} = \varepsilon_{2,t} = X_{1,t}$, for all $t \in \mathbb{Z}$, $\psi_{1,0} = 1$, $\psi_{1,j} = 0$, for all $j \neq 0$, $\psi_{2,j} = \psi_j$, $j \in \mathbb{Z}$, $\tau_1 = \tau_2 = \sigma_1^2$ and $\tau_{1,2} = 1$. Under this specification, for all $m > 0$,

$$\gamma_1(h) = \sigma_1^2 I(h = 0), \quad \gamma_2(h) = \sigma_1^2 \sum_{j \in \mathbb{Z}} \psi_j \psi_{j+h}, \quad \gamma_{1,2}(h) = \sigma_1^2 \psi_h, \quad \text{for all } h \in \mathbb{Z},$$

and since $\{X_{2,t}\}_{t \in \mathbb{Z}}$ satisfies (32), $\mathcal{K}_{k_1,k_2}(p, \tau + h, p, q + h) \longrightarrow 0$, as $h \to \infty$, for $k_1, k_2 \in \{1, 2\}$. Hence

$$\rho_{DCCA}(m) \xrightarrow{P} \frac{\psi_0 + \sum_{h=1}^{m-1} \beta_h^{(m)}(\psi_{-h} + \psi_h)}{\sqrt{\sum_{j \in \mathbb{Z}} \left[ \psi_j^2 + 2 \sum_{h=1}^{m-1} \beta_h^{(m)} \psi_j \psi_{j+h} \right]}}, \quad \text{as} \quad n \to \infty,$$

and

$$\rho_{DCCA}(m) \xrightarrow{P} \text{sign}\left( \sum_{j \in \mathbb{Z}} \psi_j \right), \quad \text{as} \quad n, m \to \infty.$$

Fig. 4(b) presents the simulation results for the scenario where the dependence is driven by an autoregressive structure as follows

$$X_{2,t} = \phi X_{2,t-1} + X_{1,t}, \quad \phi = 0.6 \quad \text{and} \quad X_{1,t} \underset{iid}{\sim} \mathcal{N}(0, 1).$$

Since $|\phi| < 1$, the coefficients $\psi_j$ defined above are such that $\psi_j = 0$, if $j < 0$, and $\psi_j = \phi^j$, if $j \geq 0$. A cumbersome but straightforward calculation yields

$$F_{DCCA}(m) \xrightarrow[n \to \infty]{P} \frac{m^3}{15(1 - \phi)(m^2 + 3m + 2)} + \frac{O(m^4)}{m^4 + 3m^3 + 2m^2} \sim \frac{m}{15(1 - \phi)},$$

$$F_{1,DFA}^2(m) \xrightarrow[n \to \infty]{P} \frac{1}{15}m + \frac{2}{15} - \frac{1}{5m} \sim \frac{1}{15}m$$

and

$$F_{2,DFA}^2(m) \xrightarrow[n \to \infty]{P} \frac{m^3}{15(1 - \phi)^2(m^2 + 3m + 2)} + \frac{O(m^4)}{m^4 + 3m^3 + 2m^2} \sim \frac{m}{15(1 - \phi)^2},$$

Hence, upon applying a Taylor expansion of $\sqrt{a + x}$ around $x = 0$, we conclude that, for any $0 < \delta \leq 1$,

$$\rho_{DCCA}(m) \xrightarrow[n \to \infty]{P} \frac{1 + o(m^{-\delta})}{\sqrt{1 + o(m^{-\delta})}} \sim \frac{1 + o(m^{-\delta})}{1 + o(m^{-\delta})} \sim 1, \quad \text{as} \quad m \to \infty.$$

*1.3.3. Dependence driven by correlated white noise processes*

Suppose that $\{X_{k,t}\}_{t \in \mathbb{Z}}$, $k \in \{1, 2\}$, are defined through (32). Then

$$\gamma_k(h) = \tau_k^2 \sum_{j \in \mathbb{Z}} \psi_{k,j} \psi_{k,j+h} \quad \gamma_{1,2}(h) = \tau_{1,2} \sum_{j \in \mathbb{Z}} \psi_{1,j} \psi_{2,j+h}, \quad h \in \mathbb{Z},$$

and since $\{X_{2,t}\}_{t \in \mathbb{Z}}$ satisfies (32), $\mathcal{K}_{k_1,k_2}(p, \tau + h, p, q + h) \longrightarrow 0$, as $h \to \infty$, for $k_1, k_2 \in \{1, 2\}$. Hence, as $n \to \infty$,

$$\rho_{DCCA}(m) \xrightarrow{P} \frac{\tau_{1,2} \sum_{j \in \mathbb{Z}} \left[ \psi_{1,j} \psi_{2,j} + \sum_{h=1}^{m-1} \beta_h^{(m)} \psi_{1,j}(\psi_{2,j-h} + \psi_{2,j+h}) \right]}{\sqrt{\tau_1^2 \sum_{j \in \mathbb{Z}} \left[ \psi_{1,j}^2 + 2 \sum_{h=1}^{m-1} \beta_h^{(m)} \psi_{1,j} \psi_{1,j+h} \right]} \sqrt{\tau_2^2 \sum_{j \in \mathbb{Z}} \left[ \psi_{2,j}^2 + 2 \sum_{h=1}^{m-1} \beta_h^{(m)} \psi_{2,j} \psi_{2,j+h} \right]}}$$

and

$$\rho_{DCCA}(m) \xrightarrow{P} \text{sign}\left( \sum_{j \in \mathbb{Z}} \psi_{1,j} \sum_{\ell \in \mathbb{Z}} \psi_{2,\ell} \right) \frac{\tau_{1,2}}{\tau_1 \tau_2}, \quad \text{as} \quad n, m \to \infty.$$

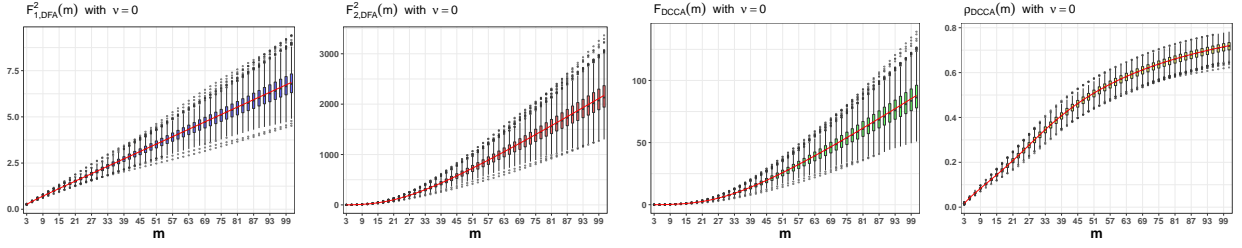

(a) Moving Average Structure with $q = 20$

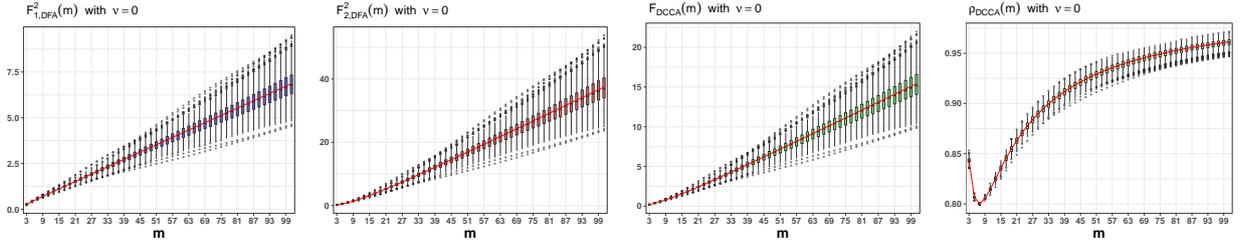

(b) Autoregressive Structure with $p = 1$

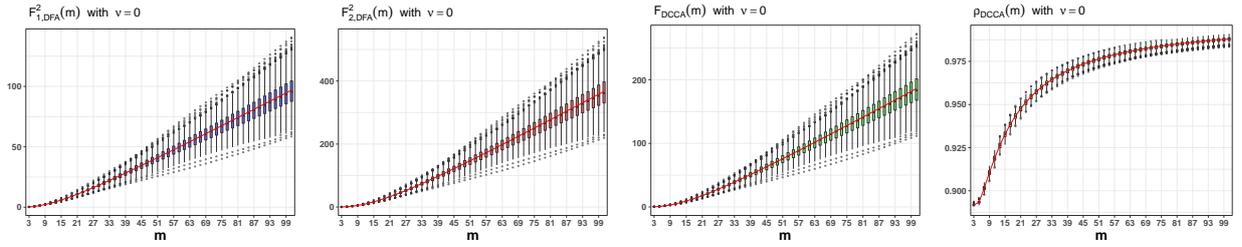

(c) Correlated white noise processes: AR marginals

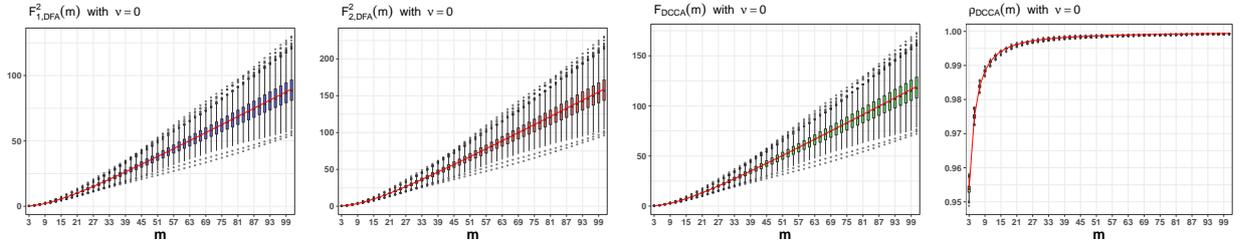

(d) Correlated white noise processes: ARMA marginals

**Fig. 4:** Boxplots considering 1,000 replications and $m \in \{3, 5, \ldots, 101\}$, for four different scenarios under which there is non-zero cross-correlation at lag $h \neq 0$. From left to right, $F^2_{1,DFA}(m)$, $F^2_{2,DFA}(m)$, $F_{DCCA}(m)$ and $\rho_{DCCA}(m)$. In all cases, the red line represents the theoretical limit obtained by letting $n \to \infty$.

Fig. 4(c) and (d) present the simulation results for two scenarios where the cross-correlated processes were generated by considering the same white noise sequence. For Fig. 4(c), the samples were generated by setting

$$X_{1,t} = 0.4 X_{1,t-1} + \varepsilon_t, \quad X_{2,t} = 0.6 X_{2,t-1} + \varepsilon_t, \quad \varepsilon_t = 0.7 \varepsilon_{t-1} + \eta_t, \quad \eta_t \underset{iid}{\sim} \mathcal{N}(0,1) \quad \text{(model 1)}$$

and, for Fig. 4(d),

$$X_{1,t} = \varepsilon_t + 0.4 \varepsilon_{t-1}, \quad X_{2,t} = \varepsilon_t + 0.6 \varepsilon_{t-1}, \quad \varepsilon_t = 0.7 \varepsilon_{t-1} + \eta_t, \quad \eta_t \underset{iid}{\sim} \mathcal{N}(0,1) \quad \text{(model 2)}.$$

7

As we shall see in the sequel, both models have a causal representation for which $\tau_1 = \tau_2 = 1$, $\tau_{1,2} = 1$ and sign$(\sum_{j\in\mathbb{Z}} \psi_{1,j} \sum_{\ell\in\mathbb{Z}} \psi_{2,\ell}) > 0$.

Let $\alpha_1 = 0.4$, $\alpha_2 = 0.6$ and $\beta = 0.7$. Then, in model 1, $\{X_{k,t}\}_{t\in\mathbb{Z}}$ can be rewritten as an AR(2) process with causal representation, $X_{k,t} = \sum_{j=0}^{\infty} \psi_{k,j}\eta_{t-j}$, where $\psi_{k,j} = (\beta^{j+1} - \alpha_k^{j+1})/(\beta - \alpha_k)$, $j \geq 0$.

Similar to the previous examples, a few algebraic manipulations yield

$$F_{DCCA}(m) \xrightarrow[n\to\infty]{P} \frac{(m^2 + 3m + 2)^{-1}m^3}{15(1-\beta)^2(1-\alpha_1)(1-\alpha_2)} + \frac{O(m^4)}{m^4 + 3m^3 + m^2} \sim \frac{m}{15(1-\beta)^2(1-\alpha_1)(1-\alpha_2)},$$

and

$$F_{k,DFA}^2(m) \xrightarrow[n\to\infty]{P} \frac{(m^2 + 3m + 2)^{-1}m^3}{15(1-\beta)^2(1-\alpha_k)^2} + \frac{O(m^4)}{m^4 + 3m^3 + m^2} \sim \frac{m}{15(1-\beta)^2(1-\alpha_k)^2}.$$

In model 2, $\{X_{k,t}\}_{t\in\mathbb{Z}}$ can be rewritten as an ARMA(1,1) for which the coefficients in the causal representation are given by $\psi_{k,j} = \beta^{j-1}(\beta + \alpha_k)$, $j \geq 0$. In this scenario we conclude that

$$F_{DCCA}(m) \xrightarrow[n\to\infty]{P} \frac{(\beta+\alpha_1)(\beta+\alpha_2)m^3}{15\beta^2(1-\beta)^2(m^2 + 3m + 2)} + \frac{O(m^4)}{m^4 + 3m^3 + m^2} \sim \frac{(\beta+\alpha_1)(\beta+\alpha_2)}{15\beta^2(1-\beta)^2}m,$$

and

$$F_{k,DFA}^2(m) \xrightarrow[n\to\infty]{P} \frac{(\beta+\alpha_k)^2 m^3}{15\beta^2(1-\beta)^2(m^2 + 3m + 2)} + \frac{O(m^4)}{m^4 + 3m^3 + m^2} \sim \frac{(\beta+\alpha_k)^2}{15\beta^2(1-\beta)^2}m,$$

Hence, upon applying a Taylor expansion of $\sqrt{a+x}$ around $x = 0$ we obtain, for model 1 and 2, for any $0 < \delta \leq 1$,

$$\rho_{DCCA}(m) \xrightarrow[n\to\infty]{P} \frac{1 + o(m^{-\delta})}{\sqrt{1 + o(m^{-\delta})}} \sim \frac{1 + o(m^{-\delta})}{1 + o(m^{-\delta})} \sim 1, \quad \text{as} \quad m \to \infty.$$

*1.3.4. Overlapping vs. non-overlapping boxes*

An interesting question is whether or not it is advantageous applying non-overlapping boxes on constructing the DCCA. The first thing to keep in mind to understand the difference between overlapping and non-overlapping boxes is that in a sample we can always fit more overlapping boxes than non-overlapping ones. For instance, in a sample of size $n = 50$, we can only fit 5 non-overlapping boxes of size 10, while we can fit 41 non-overlapping boxes of the same size. To showcase this difference, we perform a Monte Carlo simulation considering time series of size $n = 100$ presenting cross-correlation for $h \geq 0$. The dependence is driven by the moving average structure described in (3).

We simulate each time series and calculate $\rho_{DCCA}(m)$ for $m \in \{3, \cdots, 25\}$ applying overlapping and non-overlapping boxes. We replicate the experiment 1,000 times. In Fig. 5 we present the simulation results. The boxplot for each $m$ for overlapping (green) and non-overlapping (red) boxes are presented side-by-side for comparison purposes. Also presented is the true value of the DCCA given by (2). In both cases the median is very close to the true value, however, applying overlapping boxes yield estimates with significantly smaller variance in all cases. In our experiments, we also have considered a few more contexts, not presented here for the sake of brevity. Overall we found that for large samples or samples for which the underlying processes are each independent (in the context of Subsections 1.2 and 1.1), applying overlapping or non-overlapping boxes makes little difference. However, for small samples or samples presenting dependence, it is usually advantageous to apply overlapping boxes.

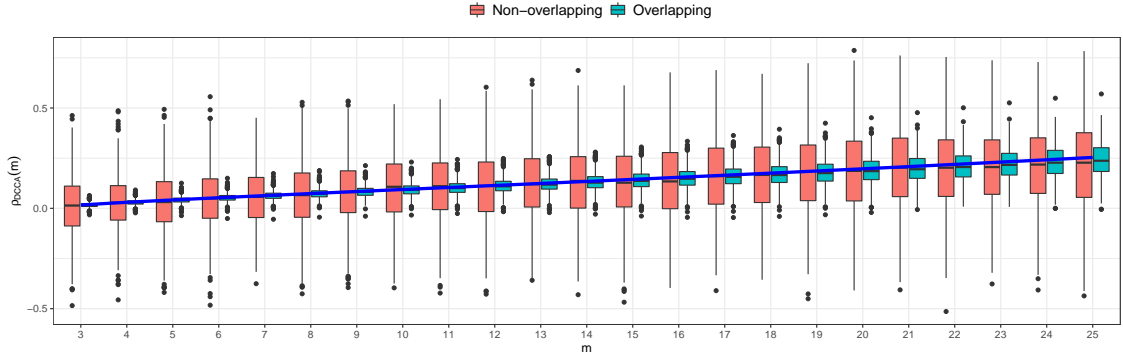

**Fig. 5:** Comparison between the $\rho_{DCCA}$ obtained employing overlapping (green) and non-overlapping (red) boxes. The blue line corresponds to the true value of the DCCA.

*1.4. Discussion of the Simulation Results*

Overall, the simulation results show that the $\rho_{DCCA}$ performs well in all scenarios, in the sense that the sample estimator behaves closely to its theoretical counterpart. This result is a consequence of the fact that $F^2_{k,DFA}(m)$ and $F_{DCCA}(m)$ are both very close to their expected values, especially when $m$ is small.

For all values of $m \in \{3, 5, \cdots, 101\}$, the median estimate of $\rho_{DCCA}$ is always very close to the theoretical values. As $m$ increases, the variances of $F^2_{k,DFA}(m)$, $F_{DCCA}(m)$, and $\rho_{DCCA}$ increase. This is expected since $F^2_{k,DFA}(m)$ and $F_{DCCA}(m)$ are averages calculated from certain quantities obtained by considering boxes of size $m + 1$. Since $m$ determines the size of the boxes, the higher the $m$, the smaller the number of boxes available, hence, the smaller the number of terms used in calculating $F^2_{k,DFA}(m)$ and $F_{DCCA}(m)$ leading to an increase in variance. On the other hand, for fixed $m$, as $n$ increases the variance decreases due to the consistency of the $\rho_{DCCA}$. This effect is shown in Fig. 2, in the context of Section 1.1, but the behavior is similar for all other cases (not shown).

We have proved (see the Appendix) that, when $\nu = 0$ and the autocovariances/cross-covariances are absolutely summable, $F^2_{k,DFA}(m)$ and $F_{DCCA}(m)$ behave asymptotically as a linear function of $m$. The simulation study shows that, for all the scenarios considered, this behavior seems to be visible even for small values of $m$.

## 2. Application

In this section, we apply the DCCA to analyze the joint behavior of the Bitcoin cryptocurrency (BTC) and 4 stock indexes, namely, S&P 500 (GSPC), Nasdaq (IXIC), Dow Jones (DJI), and Ibovespa (BVSP). The time series are obtained from Yahoo Finance website and comprehend the log-returns of the daily adjusted close data from October 18, 2010 to December 30, 2019. Since data for BTC is available everyday while data for the stock indexes is only available on trading dates and it is subject to local holidays, we only used the data available for all the five time series. This yielded an effective sample size of 2,217 after the calculation of the log-returns. Fig. 6 presents the plots of the 5 time series (left panel) and the associated scatter plots (right panel).

A well-known fact in the financial literature [see, for instance, 1] is that the log-returns of econometric time series usually show no serial autocorrelation while the absolute or squared log-returns show significant autocorrelation beyond lag 0. Often, the sample autocorrelation function (ACF) of the absolute log-returns resembles the autocorrelation structure of a long-range dependent time series and this result is used to justify the use of long memory models for the volatility. It is also well-known that the ACF is a biased estimator in the presence of long-range dependence [2]. Hence, conclusions based solely on the ACF may lead to problems in model selection.

The goal of the current analysis is to identify whether the assumption of summable correlations/cross-correlations is plausible or not for the data. Our findings may be combined with other tools to select an adequate model for the data. Since this is not the focus of our study, we shall not discuss this step here.

The first step in calculating the DFA and DCCA for the time series considered is to determine $\nu$. From Fig. 6, we observe that the time series present no trends or seasonalities. A Phillips-Perron test rejects the null hypothesis of

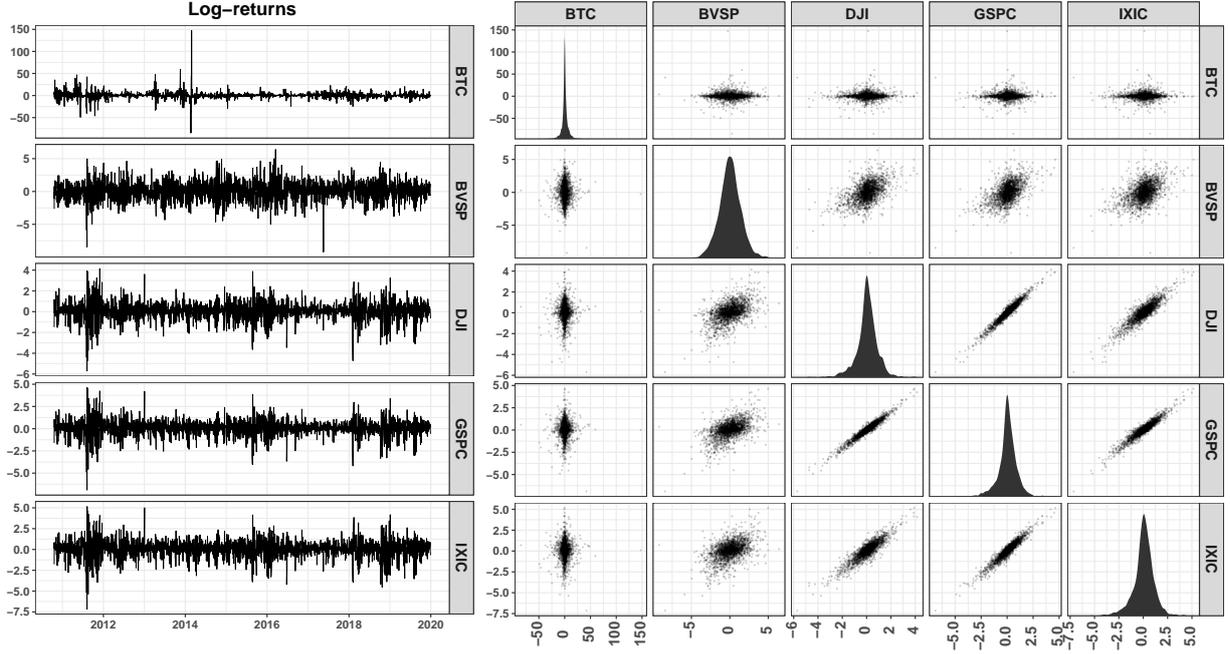

**Fig. 6:** Time series and scatter plots from the GSPC, IXIC, DJI, BVSP and BTC log-returns.

unit-root for all 5 time series ($p$-value$< 0.01$ in all cases). Hence, we conclude that $\nu = 0$ is adequate to calculate the DCCA for the log-returns and absolute log-returns. As a consequence, a local linear trend fit on each window of the integrated signal is applied to obtain the residuals necessary to calculate the DCCA. For all calculations, we consider $3 \leq m < 300$ with increments of size 2. The idea is to analyze the empirical findings in view of the theory developed.

Fig. 7 presents the graphs of $F^2_{DFA}(m)$ against $m$ (first two columns), $\ln(F^2_{DFA}(m))$ against $\ln(m)$ (third and fourth columns) and the reference line $y = x$ (in red), for the log-returns and absolute log-returns, for each time series considered. We conclude that the DFA's behave similarly for all five time series. For each time series, the $\ln(F^2_{DFA}(m))$ corresponding to the log-returns, and the reference line are parallel, indicating that the DFA is linear in $m$. According to the theory developed, this behavior is compatible with the assumption of summable autocorrelations. The same conclusion, however, does not hold for the absolute log-returns as $\ln(F^2_{DFA}(m))$ clearly present higher slope.

Fig. 8 presents the graphs of $F_{DCCA}(m)$ against $m$ (lower diagonal), $\ln(|F_{DCCA}(m)|)$ against $\ln(m)$ (upper diagonal) for all pairs (except BTC × BVSP) obtained from the GSPC, IXIC, BVSP, DJI, and BTC log-returns (green) and absolute log-returns (blue). The reference line $y = x$ (in red) was also added for comparison purposes. For the pair BTC × BVSP, the graph shows $-\ln(|F_{DCCA}(m)|)$ versus $\ln(m)$ and the reference line is $y = -x$. We conclude that the DCCA's corresponding to the log-returns and the absolute log-returns behave similarly for all pairs consisting of BTC and any stock index. In all cases, except for BTC × BVSP, even if we take into account the high variability associated with a window of size 100 or more, $\ln(|F_{DCCA}(m)|)$ and the reference line do not seem to be parallel. This result indicates that, among these four pairs, the cross-correlations are summable only for BTC × BVSP and this holds for both, the log-returns and absolute log-returns. For the remaining pairs, for the log-returns, $\ln(|F_{DCCA}(m)|)$ seems reasonably parallel to the reference line, while for the absolute log-returns, $\ln(|F_{DCCA}(m)|)$ seems to have a higher slope. These behaviors indicate that the cross-correlations are summable for the log-returns, but not for absolute log-returns.

Fig. 9 presents the values of $\rho_{DCCA}(m)$ against $m$ (lower diagonal) and $\ln(|\rho_{DCCA}(m)|)$ against $\ln(m)$ (upper diagonal) for all pairs obtained from the GSPC, IXIC, BVSP, DJI, and BTC log-returns (green) and absolute log-returns (blue). From this figure, it is clear that the time series can be grouped into three clusters, according to the $\rho_{DCCA}$'s behavior. The first cluster is formed by BTC × (BVSP, DJI, GSPC, IXIC). For this group, $\rho_{DCCA}$ is always close to 0 at $m = 3$. For BTC × BVSP, there is almost no difference between the coefficient corresponding to the log-returns and

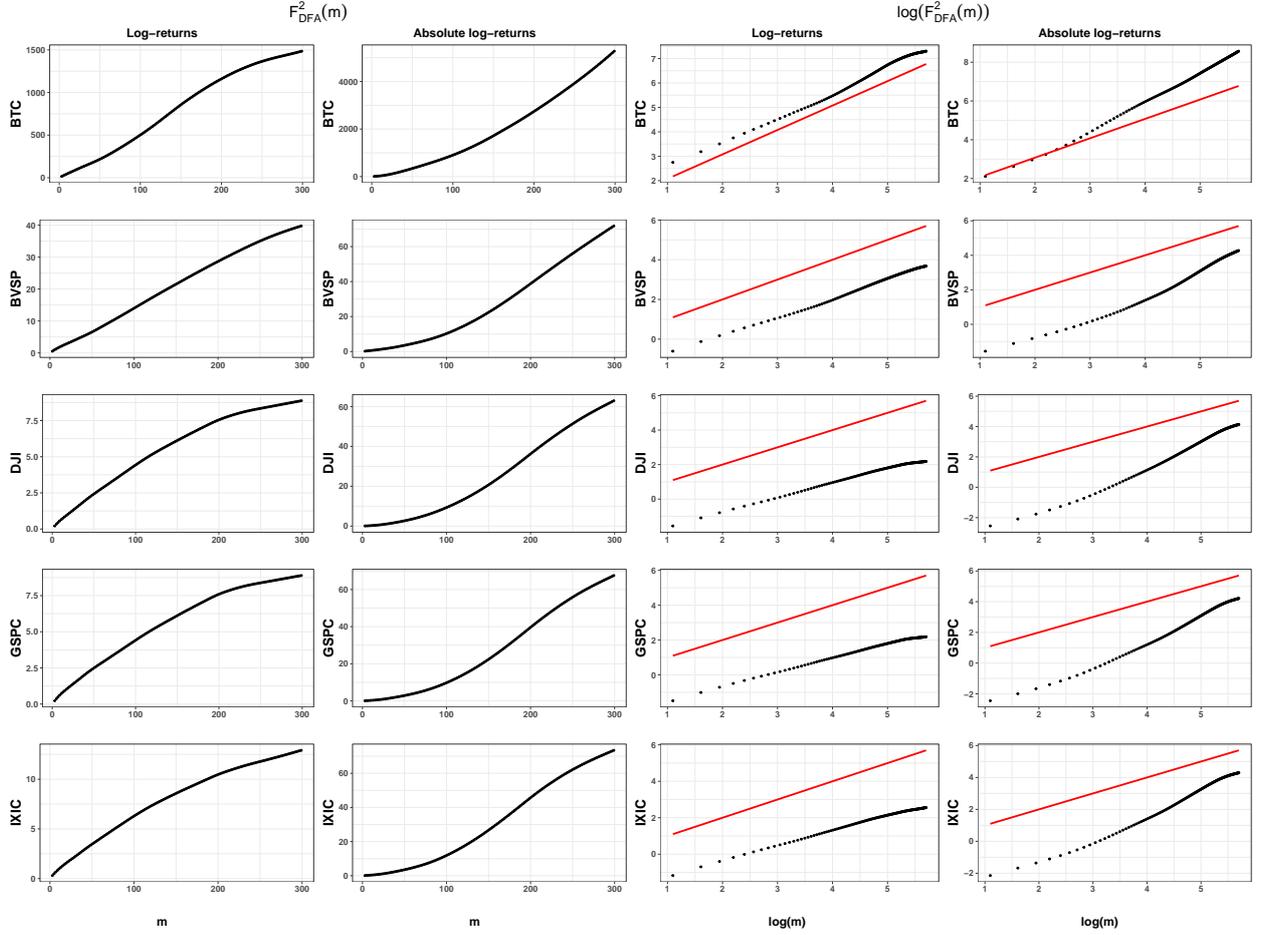

**Fig. 7:** $F^2_{DFA}(m)$ against $m$ (first and second columns) and $\ln(F^2_{DFA}(m))$ against $\ln(m)$ (third and fourth columns) for the log-returns and absolute log-returns corresponding to GSPC, IXIC, BVSP, DJI, and BTC. The red line is the reference line $y = x$

the absolute log-returns, and the values are close to zero. If we take into account the simulation results that showed high variability on the estimation of $\rho_{DCCA}$ for uncorrelated time series, this result not only supports the claim about the summability of the cross-correlations, but also indicates that the cross-correlations between the log-returns and the cross-correlations between the absolute log-returns are zero. The graphs in the logarithmic scale suggest that the coefficient corresponding to the log-returns and the absolute log-returns for the pairs BTC × (DJI, GSPC, IXIC) present similar behavior. For these pairs, $\rho_{DCCA}(m)$ increases with $m$, converging to a constant between 0.15 and 0.35, suggesting the existence of non-zero cross-correlations at lags other than 0.

The second cluster is formed by BVSP × (DJI, GSPC, IXIC). In this group, for the log-returns, $\rho_{DCCA}(m)$ is always close to 0.5 at $m = 3$ and it slowly decreases to a constant between 0.25 and 0.37. For the absolute log-returns, $\rho_{DCCA}(m)$ is close to 0.3 at $m = 3$, fast increasing until $m = 50$ and later oscillating around a constant between 0.4 and 0.5. The graphs corresponding to the log-returns, in the logarithmic scale, suggests that the time series are cross-correlated only at lag zero, since $\rho_{DCCA}(m)$ is nearly constant up to $m = 100$. After this point, the decay on these values could be due to the high variability of the estimator when the cross-correlations are non-zero only at lag 0 (see Section 1.2).

The third cluster comprises the pairs formed by DJI, GSPC, and IXIC. The behavior in this group is similar to the second one. The $\rho_{DCCA}(m)$ slowly decreases with $m$ for the log-returns and increases fast for the absolute log-returns, seemingly converging to a constant close to one. The graphs corresponding to the log returns in the logarithmic scale show behavior similar to cluster two, suggesting that cross-correlation beyond lag zero are zero. As for the absolute

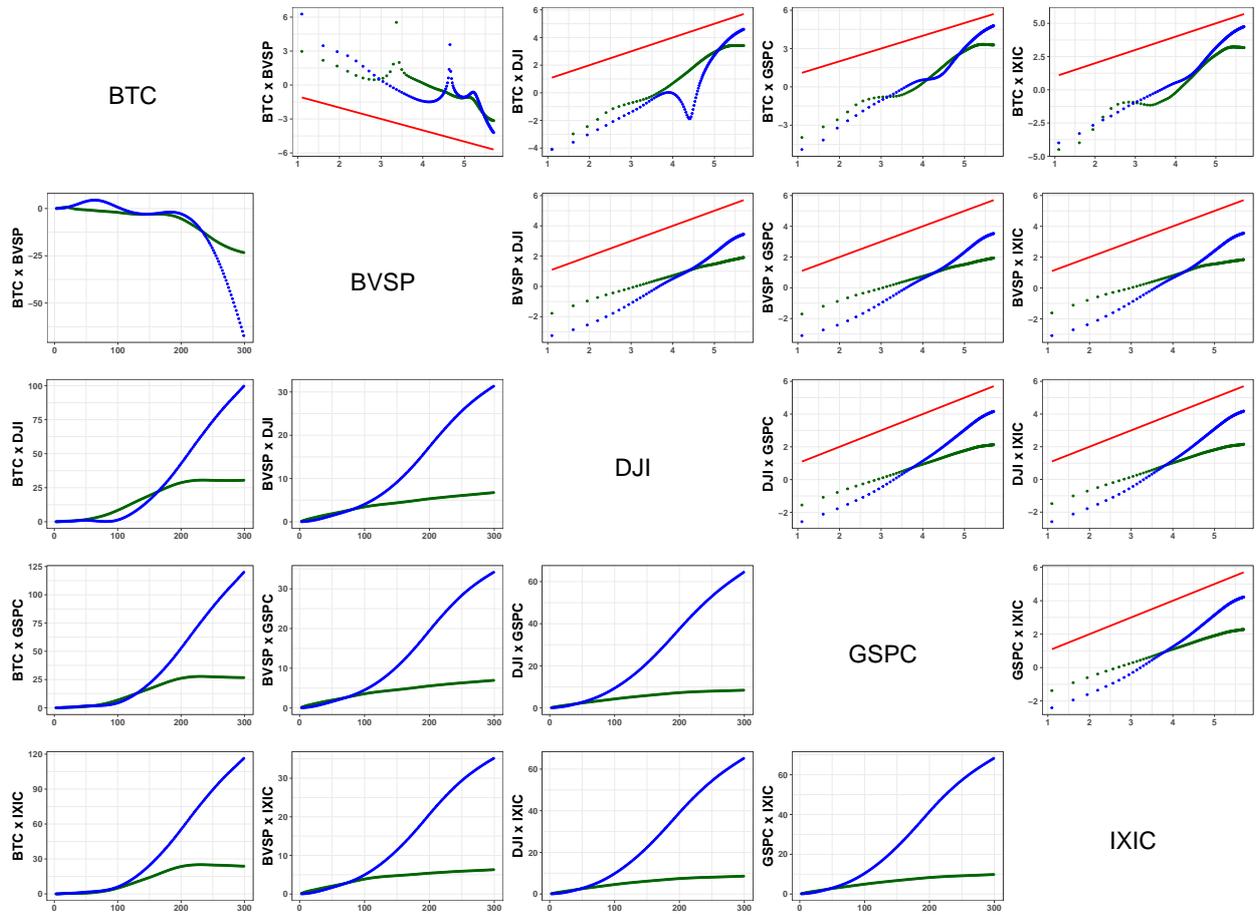

**Fig. 8:** $F_{DCCA}(m)$ against $m$ (lower diagonal) and $\ln(|F_{DCCA}(m)|)$ against $\ln(m)$ (upper diagonal) for all pairs (except BTC × BVSP) obtained from the GSPC, IXIC, BVSP, DJI, and BTC log-returns (green) and absolute log-returns (blue). For BTC × BSP the graph shows $-\ln(|F_{DCCA}(m)|)$ versus $\ln(m)$. The red line is the reference line $y = -x$, for BTC × BVSP, and $y = x$, otherwise.

log returns, the rate of convergence of $\rho_{DCCA}(m)$ to a constant seems to be slower compared to the second cluster.

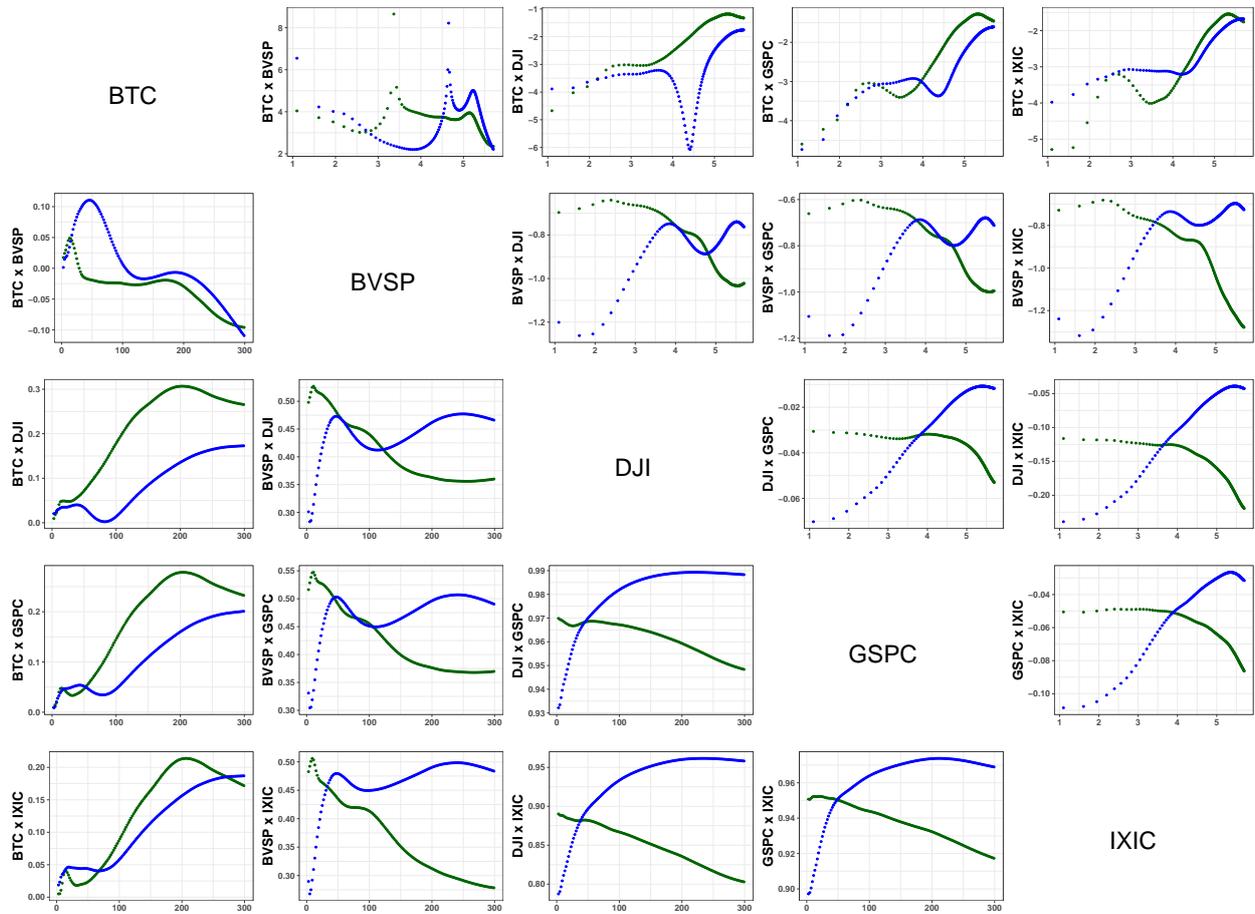

**Fig. 9:** $\rho_{DCCA}(m)$ against $m$ (lower diagonal) and $\ln(|\rho_{DCCA}(m)|)$ against $\ln(m)$ (upper diagonal) for all pairs obtained from the GSPC, IXIC, BVSP, DJI, and BTC log-returns (green) and absolute log-returns (blue).



\end{document}